\documentclass[reqno]{llncs}
\usepackage[margin=1.4in]{geometry}
\usepackage[utf8]{inputenc}
\usepackage{multirow}
\usepackage{amsmath,amsfonts,amssymb}

\usepackage{lipsum}
\usepackage{ntheorem}
\usepackage{hhline}
\usepackage{enumitem}
\usepackage{bm}

\newcommand{\R}{\mathbb{R}}

\newcommand{\lat}{{\text{lat}}}
\newcommand{\lon}{{\text{long}}}

\newtheorem*{theorem-non}{Theorem}

\usepackage{algorithm}
\usepackage{algorithmic}
\usepackage{arydshln,leftidx,mathtools}
\usepackage{blkarray}
\usepackage[dvipsnames]{xcolor}
\usepackage{tikz}
\usepackage{tikz-layers}
\usetikzlibrary{arrows,shapes,chains}
\usepackage{subcaption}
\usepackage{xcolor}
\definecolor{mydarkgreen}{RGB}{0, 158, 115}

\usepackage[colorlinks=true,breaklinks=true,bookmarks=false,urlcolor=blue,citecolor=blue,linkcolor=blue,bookmarksopen=false,draft=false]{hyperref}

\usepackage{todonotes}

\usepackage[round]{natbib}

\usepackage{setspace}

\begin{document}

\title{Evacuation Planning on Time-Expanded Networks with Integrated Wildfire Information}

\author{Steffen Borgwardt \inst{1} \orcidID{0000-0002-8069-5046} \and Nicholas Crawford \inst{2} \orcidID{0009-0009-0057-7444} \and Drew Horton \inst{3} \orcidID{0000-0003-4332-7380} \and Angela Morrison \orcidID{0000-0001-8288-5030}\inst{4} \and Emily Speakman \inst{5} \orcidID{0000-0002-7352-1355}}

\institute{
\email{\href{mailto:steffen.borgwardt@ucdenver.edu}{steffen.borgwardt@ucdenver.edu}}; \and \email{\href{mailto:nicholas.2.crawford@ucdenver.edu}{nicholas.2.crawford@ucdenver.edu}}; \and \email{\href{mailto:drew.horton@ucdenver.edu}{drew.horton@ucdenver.edu}}; \and \email{\href{mailto:angela.morrison@ucdenver.edu}{angela.morrison@ucdenver.edu}}; \and \email{\href{mailto:emily.speakman@ucdenver.edu}{emily.speakman@ucdenver.edu}};\\
University of Colorado Denver, 1201 Larimer St., Denver, Colorado 80204 United States
}

\maketitle

\begin{abstract}
We study the problem of evacuation planning for natural disasters, focusing on wildfire evacuations. By creating pre-planned evacuation routes that can be updated based on real-time data, we provide an easily adjustable approach to evacuation planning and implementation.  Our method uses publicly available data and can be tailored for a particular region or circumstance.

We formulate large-scale evacuations as maximum flow problems on time-expanded networks, in which we integrate hazard information given in the form of a shapefile. An initial flow and evacuation plan is found based on a predicted fire, and is then updated based on revised fire information received during the evacuation.

We provide a proof of concept on three locations with historic deadly fires using data available through OpenStreetMaps, a basemap for a geographic information system (GIS), on a NetworkX Python script. The results validate viable running times and quality of information for application in practice. Particular strengths are the scalability and modularity of our approach and accompanying software package. 


\end{abstract}

\section{Introduction}\label{sec:intro}
Natural disasters pose a persistent and significant threat to critical infrastructure and human lives.  Moreover, studies validate that the severity and frequency of natural disasters have intensified, with the greatest increases seen in recent decades \citep{Hooke2000,Kovacs2009,Newkirk2001}. Comprehensive preparation for these large-scale emergencies is, therefore, only becoming more imperative.  Evacuation plans are a vital element of the planning process because research consistently demonstrates that successfully evacuating a disaster region is one of the most effective strategies in mitigating loss of life \citep{Hobeika1984}. Furthermore, the {\it lack} of a well-defined evacuation plan can lead to appalling consequences; if an evacuation is not conducted efficiently, it often results in more loss \citep{Dixit2014,national_response_framework_2019}.

One sobering illustration of the chaos that can occur took place in California in 2018. Butte County suffered the ``Camp Fire", which became the deadliest and most destructive wildfire in the state's history \citep{mlmh-23,california_wert_2018}.  It resulted in 85 fatalities, in addition to destroying approximately 14,000 residences \citep{Lam2019}.  The fire started abruptly, causing major congestion on the roadways, and citizens abandoned their vehicles to continue the evacuation on foot. A key contributing factor was that the impacted area did not have sufficient road capacity on routes leaving the town.  An investigation found that as the community had grown, infrastructure was unable to accommodate the population growth.  This resulted in several roads with no exit, and only four functioning evacuation routes for the whole population \citep{stjohn_serna_lin_2018}.

Wildfire evacuations like the one in Butte County, are an example of so-called ``no-notice" evacuations. These occur when the disaster is unpredictable and requires rapid evacuation with little or no warning \citep{Chiu2007}.  This is in contrast with ``advanced-notice" evacuations, when there is enough time to warn the public about the disaster \citep{Golshani2019}.  Evacuation plans are key in both instances, but are particularly important for no-notice evacuations \citep{lindell2018}.  In this case, authorities need to communicate precise instructions to the public as quickly and as clearly as possible, and therefore, must have an initial plan ready to go \citep{Zimmerman2007}.  However, many evacuations, such as hurricanes and wildfires, play out over a significant period of time, meaning that it is advantageous if plans can also be modified based on the developing situation \citep{fema2021evacuation}.  To construct an initial plan, models can only use past empirical data to predict how the disaster is most likely to unfold. However, during an emergency, new data may indicate something unexpected, and then the designated strategy {\it must be updated as efficiently and quickly as possible}. Throughout an evacuation, timing is critical, evacuate too early, and roads may become unnecessarily congested, hindering evacuations for other regions in the affected area.  Evacuate too late and residents are left in danger \citep{badiru}.  Our method allows the user to judge the quality and feasibility of the initial evacuation plan in light of updated information and rapidly adjust when needed.

The problem of finding optimal routes for no-notice evacuations has been well-studied, see \citep{b-16,Murray-Tuite2013,Pel2012} for detailed literature reviews.  As in all mathematical modeling, evacuation models must balance the trade-off between complexity and detail. Small scale evacuations are able to take more detailed information into account such as human behavior, weather and landscape features, and traffic flow. However, when working with large road networks, it is not computationally feasible to model all these details. Model efficiency is absolutely critical if the hazard cannot be accurately predicted in advance.  The ability to rapidly update an evacuation plan during an emergency is only possible if our model is solvable within a realistic time frame, in practice this means no more than a few minutes.   

In this work, we provide an easily adjustable algorithm that provides an initial community evacuation plan and then is able to update the plan based on current hazard conditions in real time. The information provided can be utilized by officials to quickly and properly direct evacuations, identify parts of the road network that will be used, and assist in the identification of bottlenecks. 

The fundamental component of our algorithm solves the maximum flow problem on a network (possibly with multiple source and sink nodes) that has been constructed to appropriately represent the geographical region of interest. Moreover, to account for the temporal aspect of an evacuation, we adapt this network to be so-called ``time expanded" \citep{k-11}, and further adjust the structure to account for changes caused by a developing hazard. Solutions to the maximum flow problem on a time-expanded network are already a critical tool in the evacuation planning literature, and our contribution builds on past work in a number of ways.
\begin{itemize}
    \item We provide a freely available and user-friendly self-contained software package, requiring minimal computational resources and using open-source data, which allows both experts and non-experts to produce and visualize evacuation plans for any geographic region covered by Open Street Maps \citep{openstreetmap}.
    \item The package may be used to produce initial evacuation plans for multiple and varying scenarios ahead of time which can be used to assist in emergency planning initiatives.  Moreover, the model runs quickly enough that the user can update plans to account for any changes in real-time.
    \item The package is highly-adaptable.  For example, in our test instances, we model time to the minute and model the road network accurate to 10 meters ($\approx$ 30 feet).    However, should the user require an alternative balance between fine detail and computational feasibility, these parameters can be easily changed. 
    \item Our package includes a procedure to model the development of a hazard over time, allowing us to take into account factors such as decreased roadway capacity and the potential for blocked roads due to evolving conditions. However, the modular format allows a user to instead easily incorporate data from more sophisticated models if this is available. Hence, our package can be used for accurate evacuation planning in multiple hazards and scenarios. 
\end{itemize}

\section{Preliminaries and Contribution}\label{sec:prelim}
In this section, we briefly present some necessary background and notation.  Specifically, we address dynamic flow problems and the literature on evacuation planning. We conclude this section by presenting our contributions in the context of previous work.

\subsection{Dynamic Flows in Time-Expanded Networks}\label{sec:dynamicflows}

A dynamic network is a network with node set $N$, and arc set $A \subseteq N \times N$, such that the capacity and travel time information associated with each arc may vary over a given discretized time horizon, $[T]_0=\{0, 1, \dots, T\}$.  Consider the multi-source and multi-sink connected dynamic network $G = (N,A,T)$ equipped with the following parameters and variables:

\begin{align*}
\text{Parameters} & \\
    &S:\text{set of sources } s_k\\
    &D:\text {set of sinks } d_\ell\\
    &\lambda_{ij}(t)\in \mathbb{N}: \text{ travel time along arc } (i,j) \text{ at time } t \\
    &u_{ij}(t)\in \mathbb{N}: \text{ upper capacity of arc } (i,j) \text{ at time } t \\
    &a_{i}(t)\in \mathbb{N}: \text{ waiting capacity of node } i \text{ at time } t\\
    \text{Variables} & \\
  &x_{ij}(t)\in \mathbb{N}: \text{ movement flow on arc } (i,j) \text{ at time } t\\
  &x_{ii}(t)\in \mathbb{N}: \text{ holdover flow of node } i \text{ at time } t
\end{align*}

Further, we define the following sets:
\begin{align*}
    &Q(t)=\{t':t = t'+\lambda_{id}(t')\} \\
    &R(t)=\{(i,j) \in A: t+\lambda_{ij}(t)\leq T\}
\end{align*}

The set $Q(t)$ contains the time instances $t' < t$ in which there exists an arc $(i,d) \in A$ such that flow along the arc reaches $d$ at precisely time $t$. The set $R(t)$ is the set of all arcs $(i,j) \in A$ whose travel times $\lambda_{ij}(t)$ plus current time $t$ do not exceed the time horizon $T$, i.e., people who begin traveling along arc $(i,j)$ at time $t$ can reach their destination before or at time horizon $T$.  Using this notation, we formally state the general {\it dynamic maximum flow problem} on $G$ as follows.

\renewcommand\arraystretch{2.8}
\begin{align}
     \text{max } \qquad &\displaystyle \sum_{d \in D}\displaystyle \sum_{(i,d) \in A} \sum_{t\in [T]_0} \sum_{t' \in Q(t)} x_{id}(t') && \label{obj} \\
    \text{s. t. } \qquad &\displaystyle{\sum_{(j,i)\in A} \sum_{t' \in Q(t)}} x_{ji}(t') \;\; -  && \nonumber \\
    & \qquad \quad  \displaystyle{\sum_{(i,j) \in R(t)}}x_{ij}(t) = x_{ii}(t)-x_{ii}(t-1)  &&\forall i \in N \setminus \{S\cup D\}, \ \forall t \in [T]_0\label{flowbal} \\
    0 \;\; \leq \;\;  &{x_{ij}(t)} \;\; \leq \;\; u_{ij}(t)   &&\forall (i,j) \in A, \ \forall t \in  [T]_0 \label{arccap}\\
    0 \;\; \leq\;\;  &{x_{ii}(t)} \;\; \leq \;\; a_{i}(t)    &&\forall i \in N, \ \forall  t \in [T]_0. \label{waitcap}\\
\tag{DMF LP}
\label{Dynamic}
\end{align}

The objective function, (\ref{obj}), maximizes the amount of flow that reaches the sinks $d \in D \subset N$ over all time instances $t \leq T$. The first set of constraints, (\ref{flowbal}), represent flow balance in the dynamic setting. That is, for every node $i \in N \setminus \{S,D\}$ and time instance $t \in [T]_0$, the amount of in-flow and out-flow must be equal. The remaining constraints, (\ref{arccap}),(\ref{waitcap}), bound the capacity of movement and holdover flow, respectively.

For computational purposes, it is often useful to transform a dynamic network into a time-expanded network: a single static network with copies of the node set for each time instance, and new arcs constructed from information based on the travel time of each arc in the dynamic network. In addition to being convenient, this removes the need to store arc travel time information explicitly. The nodes in the time-expanded network do not need to hold time-dependent information.

We associate the nodes in the time-expanded network with the nodes in the original network in the following way.  We label each copy of node $i\in N$ in the time-expanded network as $i(t)$ using the  formula
\begin{equation}
    i(t) = i + |N|\cdot t,
\end{equation}
where $t \in [T]_0$. This gives us the following definition for the nodes on the time-expanded network,
\begin{equation}
    N_{T} := \{i(t): i \in N, t \in [T]_0\}.
\end{equation}

For an example, see Figure \ref{fig:time_expanded_net}.  Here, the copies of node 1 from the original dynamic network (\subref{fig:time_expanded_neta}) correspond to nodes 1, 4, 7, and 10 in the time-expanded network (\subref{fig:time_expanded_netb}). Arcs in the time-expanded network, which are separated into {\em movement arcs}, $A^M$, and {\em holdover arcs}, $A^H$, can be identified similarly. Arc $(2,3) \in A$ of the original dynamic network has a travel time of 2 for all time instances. Copies of arc $(2,3) \in A$, called movement arcs, correspond to arcs $(2,9)$ and $(5,12)$ in the time-expanded network. Lastly, to represent holdover flow, arcs are added between nodes $i(t)$ and $i(t+1)$. For example,  the arcs $(1,4)$, $(4,7)$, and $(7,10)$ are included in the time-expanded network for the holdover at node $1 \in N$. In general, we can define these movement and holdover arcs as follows,

\begin{equation}
    A^M := \{(i(t),j(t')): (i,j) \in A, \quad t' = t+\lambda_{ij}(t) \leq T\}
\end{equation}  
\begin{equation}
    A^H := \{(i(t),i(t+1)): i\in (S \cup D), \quad t \in [T-1]_0\}.
\end{equation}

Finally, we define the sets $S_T$ and $D_T$ as the sets of nodes in the time-expanded network that correspond to the source set $S$ and sink set $D$ of the dynamic network, respectively.  We have
\begin{align}
    &S_T=\{\ s_{i}(0) \in N_T: s_i \in S, s_{i}(0) = s_{i}\} \\
    &D_T=\{\ d_{j}(T) \in N_T: d_j \in D, d_{j}(T) = d_{j} +|N|\cdot T \}.
\end{align}
Sources, $s_i(0)$, of the time-expanded network are all the copies of each source node, $s_i \in N$, from the original dynamic network that appear in the first time instance, $t=0$. The sink nodes of the time-expanded network are the final copy, $d_j(T)$, of each of the sink nodes, $d_j \in D$, of the dynamic network. These nodes are important to distinguish as they are the only nodes where we chose to allow holdover in our definition of a time-expanded network. 
The formulation could easily be adjusted for other nodes to have holdover.

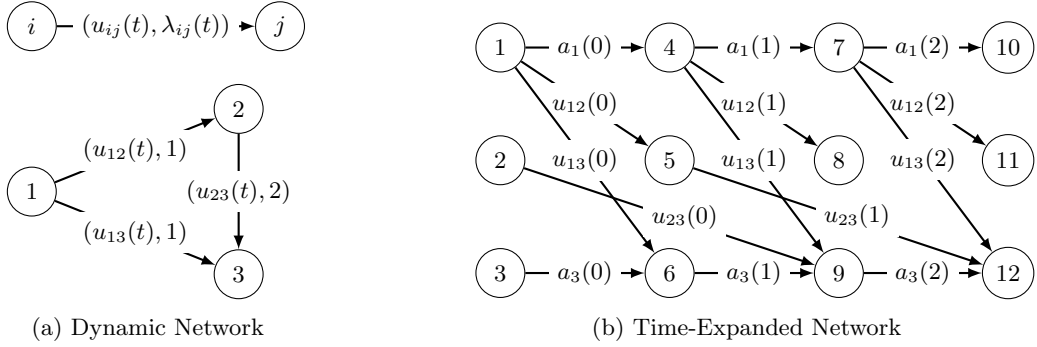
\begin{figure}[htb!]
     \centering
     \begin{subfigure}[b]{0.3\textwidth}
         \centering
         \subcaptionbox{Dynamic Network  \label{fig:time_expanded_neta}}{
\begin{tikzpicture}[scale = 0.65]
\tikzset{vertex/.style = {shape=circle,draw,minimum size=2em}}
\tikzset{edge/.style = {-triangle 90,fill=black}}
\tikzset{edgedd/.style = {dashed,-triangle 90,fill=black}}
\node[vertex] (i) at (-4,4) {$i$};
\node[vertex] (j) at (2,4) {$j$};
\node[vertex] (1) at (-4,0) {1};
\node[vertex] (2) at (1,2) {2};
\node[vertex] (3) at (1,-2) {3};
\draw[-latex,thick] (1) to node[midway, fill=white] {$(u_{12}(t),1)$} (2);
\draw[-latex,thick] (1) to node[midway, fill=white] {$(u_{13}(t),1)$} (3);
\draw[-latex,thick] (2) to node[midway, fill=white] {$(u_{23}(t),2)$} (3);
\draw[-latex,thick] (i) to node[midway, fill=white] {$(u_{ij}(t),\lambda_{ij}(t))$} (j);
\end{tikzpicture}}
     \end{subfigure}
     \hfill
     \begin{subfigure}[b]{0.6\textwidth}
         \centering
         \subcaptionbox{Time-Expanded Network \label{fig:time_expanded_netb}}{
\begin{tikzpicture}[scale = 0.85]
\tikzset{vertex/.style = {shape=circle,draw,minimum size=2em}}
\node[vertex] (1) at (0,2) {1};
\node[vertex] (1') at (3,2) {4};
\node[vertex] (1'') at (6,2) {7};
\node[vertex] (1''') at (9,2) {10};
\node[vertex] (2) at (0,0) {2};
\node[vertex] (2') at (3,0) {5};
\node[vertex] (2'') at (6,0) {8};
\node[vertex] (2''') at (9,0) {11};
\node[vertex] (3) at (0,-2) {3};
\node[vertex] (3') at (3,-2) {6};
\node[vertex] (3'') at (6,-2) {9};
\node[vertex] (3''') at (9,-2) {12};
\draw[-latex,thick] (3) to node[midway, fill=white] {$a_3(0)$} (3');
\draw[-latex,thick] (3') to node[midway, fill=white] {$a_3(1)$} (3'');
\draw[-latex,thick] (3'') to node[midway, fill=white] {$a_3$(2)} (3''');
\draw[-latex,thick] (1) to node[midway, fill=white] {$a_1(0)$} (1');
\draw[-latex,thick] (1') to node[midway, fill=white] {$a_1(1)$} (1'');
\draw[-latex,thick] (1'') to node[midway, fill=white] {$a_1$(2)} (1''');
\draw[-latex,thick] (1) to node[midway, fill=white] {$u_{12}(0)$} (2');
\draw[-latex,thick] (1') to node[midway, fill=white] {$u_{12}(1)$} (2'');
\draw[-latex,thick] (1'') to node[midway, fill=white] {$u_{12}(2)$} (2''');
\draw[-latex,thick] (1) to node[midway, fill=white] {$u_{13}(0)$} (3');
\draw[-latex,thick] (1') to node[midway, fill=white] {$u_{13}(1)$} (3'');
\draw[-latex,thick] (1'') to node[midway, fill=white] {$u_{13}(2)$} (3''');
\draw[-latex,thick] (2) to node[midway, fill=white, xshift=2mm] {$u_{23}(0)$} (3'');
\draw[-latex,thick] (2') to node[midway, fill=white, xshift = 2.5mm] {$u_{23}(1)$} (3''');
\end{tikzpicture}}
     \end{subfigure}
        \caption{A dynamic network and an associated time-expanded network.}
        \label{fig:time_expanded_net}
\end{figure}

We can now formally state the transformation as follows. 

\begin{definition}[Time-Expanded Network]\label{def:ten}
    The time-expanded network (TEN) $G_{T} = (N_{T},A_{T})$ of a dynamic network $G = (N,A,T)$ over a time horizon $T$ is defined by
 \begin{equation}
     N_{T} := \{i(t): i \in N, t \in [T]_0\} \text{ and } A_{T} := A^{M} \cup A^{H},
 \end{equation}
 where
    \begin{align*}
        A^M &:= \{(i(t),j(t')): (i,j) \in A, \quad t' = t+\lambda_{ij}(t) \leq T\}\\
        A^H &:= \{(i(t),i(t+1)): i\in (S \cup D), \quad t \in [T-1]_0\},
    \end{align*}
for the movement arcs, $A^M$, and holdover arcs, $A^H$, respectively. The capacities $u_{i(t)j(t')}$ and $u_{i(t)i(t+1)}$ for movement and holdover arcs respectively are set as
\begin{align}
    &&u_{i(t)j(t')} &:= u_{ij}(t) && \text{for } i,j \in A,t' = t+\lambda_{ij}(t), t \in [T-1]_0 &&\\
    &&u_{i(t)i(t+1)} &:= a_i(t) && \text{for } i \in N , t \in [T-1]_0.&&
\end{align}
\end{definition}

Using Definition \ref{def:ten}, we can now reformulate the dynamic maximum flow problem on $G$, i.e., (\ref{Dynamic}), as a standard static multi-source, multi-sink maximum flow problem on the time-expanded network $G_{T}$ \citep{k-11}:

\begin{align}
     \text{max } \qquad \quad \quad &\displaystyle{\sum_{d \in D_T} \sum_{ (i,d) \in A_{T}} x_{id}}  &&\\
    \text{s. t. } \qquad \quad \quad&\displaystyle{\sum_{(j,i)\in A_{T}}}x_{ji} - \displaystyle{\sum_{(i,j) \in A_{T}}}x_{ij} = 0 &\quad  &\forall i \in N_{T} \setminus \left(S_T \cup D_T \right)\\
    &0 \;\;\leq \;\;x_{ij}\;\; \leq \;\; u_{ij}  &\quad  &\forall (i,j) \in A_{T}.   \\
\label{TEN_LP}
\tag{TEN LP}
\end{align}

We conclude by observing the efficiency of constructing the TEN. Recall that, in the arithmetic model of computation, a unit cost is associated to arithmetic and set operations regardless of size of numbers, memory access, or control flow overhead.

\begin{proposition}\label{prop:setup}
Let (\ref{Dynamic}) represent a dynamic flow problem on a dynamic network $G=(N,A,T)$. The associated TEN $G_T=(N_T,A_T)$ satisfies $|N_T|= (T+1)|N|$ and $|A_T|\leq T(|A|+|N|)$, and can be constructed in linear time $O(T|A|)$ in the arithmetic model of computation.
\end{proposition}

\begin{proof}
First, we confirm the size of the TEN. The number of nodes $|N_T|$ is readily seen to be exactly $(T+1)|N|$ by definition. The set of movement arcs satisfies $|A^{M}|\leq T|A|$, as at most one arc is added to $A^{M}$ for each combination of $(i,j)\in A$ and $t\in [T-1]_0$. The set of holdover arcs satisfies $|A^{H}| = T|S\cup D| \leq T|N|$, as one arc is added to $A^{H}$ for each combination of $i \in (S \cup D)$ and $t\in [T-1]_0$. Overall, the size of  $A_{T} := A^{M} \cup A^{H}$ is bounded above by $|A_{T}| \leq T(|A|+|N|)$.

It remains to verify the efficiency of construction. Construction of the set of nodes $N_{T}$ requires constant effort for each of the $(T+1)|N|$ nodes. The sets $S_T$ and $D_T$ can be constructed at constant effort for each of their nodes during the setup of the first and final time instance nodes in $N_{T}$ by identifying whether they correspond to source or sink nodes in $S$ or $D$, respectively. Together the sets $N_{T}$, $S_T$, and $D_T$ are constructed in time $O(T|N|)$.

Next, we consider the construction of the set of arcs $A_{T} := A^{M} \cup A^{H}$. For each $(i,j) \in A$ and each $t$, the property $t' = t+\lambda_{ij}(t) \leq T$ is checked in constant time to decide whether to add $(i(t),j(t'))$ to $A^M$. For each $i\in (S \cup D)$ and $t\in [T-1]_0$, the arc $(i(t),i(t+1))$ is added. Together, the set $A_{T}$ is constructed in time $O(|A|T+(|S\cup D|)T)$. The capacities of both types of arcs are directly copied from the input, which can be done at constant effort for each during setup of the arcs sets.

We arrive at an overall running time in $O(T(|A|+|N|))$ for setup of TEN $G_T$, and $O(T(|A|+|N|))\subset O(T|A|)$ as the network is connected. Recall that the input for (\ref{Dynamic}) specifies travel times and capacities for all arcs in $A$, for all time instances in $[T]_0$. Thus, $T|A|$ is linear in the input size for (\ref{Dynamic}). \qed
\end{proof}

It is clear that the value $T$ plays a large role in the problem size: both the number of arcs and the number of nodes scale linearly with $|T|$. To avoid computational overhead, it is of interest to use a minimal or small $T$ that allows a feasible solution.

\subsection{Evacuation Models}
\label{sec:evacmodels}

Evacuation models have been extensively studied, with the various approaches often categorized as either macroscopic or microscopic.  Macroscopic models treat the crowd as a whole and represent their behavior as a flow rather than focusing on each person or vehicle as an individual. Macroscopic models address collective crowd behavior through traffic flow theories \citep{Lovas1994}, fluid dynamics \citep{Henderson1974}, and continuum mechanics \citep{Hughes2002}. They use aggregate parameters such as density, flow rate, and average speed. These approaches typically require modeling simplifications, such as converting the evacuation process to a target function \citep{Bretschneider2011} or representing the evacuation area as a network of links and nodes \citep{Lovas1994}.

In contrast, microscopic models consider individual behavior on a more detailed level, and can account for interactions between people or vehicles in the evacuation \citep{Helbing1995, Pan2006}. These models employ various techniques including Social Force Models \citep{Helbing1995}, Cellular Automata \citep{Burstedde2001,Pelechano2008}, Lattice Gas approaches \citep{Helbing2003}, and Agent-Based Models \citep{Shi2009,Penn2002}.

The type of large-scale emergencies we consider, such as wildfires and tsunamis, necessitate the use of long-distance regional evacuation plans as opposed to small-scale plans, which are appropriate in a building fire, for example \citep{lp-13}.  Regional evacuations lead to the displacement of a large number of people using personal vehicles or other means of long-distance transportation, and due to the complexity, require the use of macroscopic evacuation models. 

Given this context, stage-based evacuation can be a useful approach.  Here, emergency situations are managed by systematically prioritizing evacuation zones to minimize congestion and optimize resource utilization. The solution framework divides the affected network into zones based on their proximity to the hazard. This classification is important as zones closer to the incident location generally face shorter safety windows and experience more severe impacts \citep{Liu2006}.   While not directly addressed in this paper, our work can be easily extended for use within a stage-based model. 

The literature on stage-based evacuation has contributed various methodological developments. For example, \cite{Liu2006} presents a  model that focuses on mitigating network congestion. This can be achieved in various ways depending on the specific instance, for example, having crowds exit through a limited number of gates \citep{Lin2008}, or reducing en-route wait time by adding delays to the departure announcement of different segments \citep{Li2012}. Other models prioritize the time aspect of evacuations by minimizing the total evacuation time \citep{Chien2007}, or by optimizing the warning time given to residents in varying evacuation stages \citep{Chiu2007}. 
Many specific contexts have also been studied, for example, \cite{Saadatseresht2009} and \cite{Stepanov2009} consider multi-objective optimization within an urban setting.  Objective functions considered include clearance time, total traveled distance, and blocking probability for evacuation planning. For additional models and information on stage-based evacuation, see \citep{Li2012,Li2018,Hosseini2021}.

 In large-scale evacuations, such as those we consider in this work, several factors determine a successful outcome, including warning time, response time, evacuation routes, and how traffic flow is assigned to specific routes \citep{Dash2007,Lindell2007}.  Traffic assignment models \citep{Dafermos1969} can be considered within the framework of User Equilibrium (UE) models, System Optimal (SO) models, Nearest Allocation (NA) models, and Constrained System Optimal (CSO) models. 
An evacuation plan is said to have User Equilibrium when no traveler can improve their individual travel time by changing routes (away from what the model prescribed) \citep{Sheffi1985}. This desirable property may not be guaranteed when taking a purely System Optimal approach because here we maximize the \emph{total} benefit of the system rather than considering the benefit to each individual separately. As such, some travelers may be assigned to routes that are much longer than the ones they would prefer to take. In contrast, when an evacuee uses a shortest path to reach their nearest shelter (based on geographical distance), this is considered a Nearest Allocation approach. Constrained System Optimal models seek the best of both worlds by combining SO with UE and NA to find an optimal solution for the system as a whole but given further constraints.  The additional restrictions ensure that drivers are only assigned to ``acceptable" or reasonable paths. For example, we could enforce that no driver should be assigned to a route that takes more than $x$ minutes longer than their shortest path to the given shelter.  For a more detailed explanation of these properties in the context of evacuation planning see \citep{Peeta2001,Sheffi1985,Wardrop1952}.

Within these traffic modeling approaches, we can consider both the static and dynamic modeling of traffic flows \citep{highway_capacity_manual_2010}. In the case of static flow, we assume that the traffic flows are fixed for the entire planning horizon whereas a dynamic model allows flow rates to change over time.  This is useful if we expect our hazard to impact the condition of our roads, for example. The common dynamic approaches are dynamic system optimal (DSO), dynamic user equilibrium (DUE), and dynamic traffic assignment (DTA).

DTA models are optimization- and/or simulation-based models that are able to represent the time-varying characteristics of traffic conditions during an evacuation. Elements that have been considered within a DTA model framework for long-distance evacuations include maximum flow models \citep{gdg-16}, shortest path models \citep{wth-11}, and cell-based transmission models \citep{lyl-06}. Cell-based models (CTM) are a form of DTA which predict macroscopic traffic behavior by evaluating flow and density at a finite number of points and time steps along an arc. CTMs can be modeled as a linear program; however, due to the excessive number of constraints and variables for real-world networks, they require a prohibitive amount of memory and CPU time \citep{Kimms2011,Zheng2011}.  Dynamic maximum flow models can also be modeled as linear programs and solved accordingly.  However, they are more commonly solved using a specialized  algorithm such as Ford-Fulkerson \citep{ff-58}.  For an in-depth review of general DTA models, we refer the reader to \citep{Peeta2001}.  

Most of the DTA models in the literature, such as \citep{Bish2013,Chiu2007,Liu2006b,Ng2010,Yazici2010}, do not scale well and are not tractable for applications with large networks. A generally preferred methodology is to describe the static network at different points of time, i.e., to use a time-expanded version of the original network (TEN), see Definition \ref{def:ten} in Section \ref{sec:dynamicflows}. There are several papers which employ dynamic network flow models by constructing the TEN from the original static network \citep{Bretschneider2011,Hamacher2013,Hamacher2002,Lim2012}, and several versions of dynamic network flow problems are considered for evacuation planning.  These include dynamic maximum flow, earliest arrival flow, quickest flow, and dynamic minimum cost flow \citep{b-16}.

In \citep{Lim2012}, the authors model short-notice evacuations as a capacitated network flow problem on a dynamic time-expanded network.  They develop an ``evacuation scheduling algorithm" to obtain an evacuation schedule for each impacted region including prescribed evacuation paths and flow on the routes. The heuristic method regularly achieves high-quality solutions for even large networks on which exact optimal solutions are too computationally expensive to obtain. The model determines the number of people that can be evacuated and the order in which to send people to the safe locations they have been paired with. The granularity of the time interval used is 15 minutes. In this work, we build on several of the ideas from the evacuation literature presented in this section.

\subsection{Our Contribution}\label{sec:contributions}

We provide a method for creating dynamic networks based on location input, a predictive hazard model, techniques for constructing initial and updated evacuation plans, and a means to extract interpretable information, for example, a visualization of the roads used in an evacuation. We directly incorporate hazard information into a time-expanded network and allow for fine-grain detail both in the density of the network used and the amount of time captured at each interval. Our work provides users with an entire package for constructing evacuation plans. Moreover, each piece of our package is interchangeable with other methods one may wish to use, allowing the incorporation of user expertise and existing computing resources.

The computational package allows a user to specify an area of interest, for example, a city or county, along with sub-areas that need to be evacuated to designated safe zones. It then provides an evacuation plan for the area, tailored for a specified hazard, which originates in given geographic areas and grows over time. Source code for the algorithm and examples of case studies can be found at \url{https://github.com/angela-r-morrison/fire_evac}.
\vspace{2mm}

Our method can be broken down into the following steps:
\begin{enumerate}
    \item Pre-processing of public data to create a network representation of the required area.
    \item An initial model of a hazard, and specifically, its geographical spread over time.
    \item Construction of an initial evacuation plan using the hazard model from (2).
    \item Revised evacuation plans to be implemented at a user-specified time, calculated from a hazard model updated using real-time information.
\end{enumerate}

A feature of our method is the ability to provide not only an initial evacuation plan but updated evacuation plans using real-time data as it becomes available. Each evacuation plan is created by solving a series of maximum flow problems on a time-expanded network that has been precisely constructed to take the geographical spread of our given hazard into account.  

In Section \ref{HTEN}, we describe the process of integrating a hazard into a dynamic network and define the result to be a {\it wildfire time-expanded network} (WTEN). Time is discretized, with each time instance representing one minute. Our algorithm builds on several functionalities of the Python package Shapely to precisely represent the area(s) that are impacted in each time instance and saves this data as a shapefile. The information is combined with the dynamic network to identify which parts of the network become unreachable at at each time instance, as well as to appropriately adjust capacities of roads or arcs that are close to the hazard. Our algorithm is able to import any shapefile, and thus, is compatible with sophisticated wildfire models.   In our proof of concept (Section \ref{sec:test_cases}), we construct a simple fire set as the union of circles expanding linearly over time.  We validate the efficiency of the WTEN construction, parameterized with the complexity of distance calculations between roads and the fire set.

Once constructed, we solve a maximum flow problem on the WTEN. This can be formulated as a linear program, (see \ref{WTEN_LP}), and for a convenient wording, we say that we solve (\ref{WTEN_LP}) even though in practice we are using a more efficient combinatorial network flow algorithm.  While there do exist combinatorial algorithms for solving problems on the original dynamic network, their efficiency relies on certain restrictions. These restrictions could force the oversimplification of large-scale evacuation models. For example, simply adding or removing arcs cannot capture changes in arc capacity or travel time which are needed to represent changing road conditions during an evacuation.  This serves as additional motivation for utilizing time-expanded networks in our setting.  

In Section \ref{sec:alg}, we provide a high-level explanation of each step in our algorithm. First, we describe the pre-processing of public network data used to construct the WTEN. This includes using the Python packages OSMNX, \citep{OSMNX-17}, and NetworkX. We start by taking data from OpenStreetMaps \citep{openstreetmap} and create a network object. Node and arc information such as road capacities, supply and demand values (which represent the population we need to evacuate and the capacity of evacuation locations), and longitudinal/latitudinal coordinates are added to the network.
Next, a sample fire is created and used to construct a WTEN; see Section \ref{HTEN}. We find an optimal solution to (\ref{WTEN_LP}) using the shortest augmenting path algorithm. This solution serves as an initial evacuation plan. We confirm the overall running time and prove correctness, in particular correct termination, of the approach.

Next, we describe an update of this initial plan. We create updated fire data starting some time instance into an ongoing evacuation. This mimics a change in the actual fire (and/or initial fire prediction) realized in real-time. We model the new fire to be worse than the original prediction; either with accelerated growth or starting at additional locations.  We update the initial evacuation plan, starting at the time instance when the new prediction is relayed and a new plan can be implemented. This means that flow is fixed for past time instances, but is changed and reoptimized several time instances even before the predicted change. This `gap' is the reason why the computational efficiency of our approach is important: the faster the computation, the earlier the implementation of a new plan can begin. Efficiency and correctness of the updated plan computation transfers readily. Moreover, we observe that for typical communities, our computation times are only one to two minutes, leading to minimal delay.

In Section \ref{sec:test_cases}, we describe three case studies. The locations are chosen based on recent deadly wildfires or known areas of high fire danger, and to test our model on varying infrastructure. Case 1 (Lyons, Colorado), provides an example of a small rural network. Case 2 (Butte County, California) serves as a test for a large-scale rural network. Finally, Case 3 (Alexandroupolis, Greece) serves as a test for urban road networks. For Cases 1 and 3, we demonstrate the update portion of the algorithm. Case 2 was designed to model the Camp Fire closely in terms of fire location and growth, so all computations are based on the historic spread of the fire and we do not consider updates.  We study running times and computational efficiency of the different portions of the algorithm, and discuss the output format and deliverables we provide for decision makers.  Finally, we conclude in Section \ref{sec:conclusion} with some closing remarks on the strengths and challenges of our work, as well as potential future improvements. 

\section{Integration of Hazard Information in Time-Expanded Networks}
\label{HTEN}

Our modeling goal is the integration of dynamic hazard information, in particular wildfire information, into a time-expanded network. More specifically, we aim to construct a modular means of integrating information on any hazard into a time-expanded network. As a proof of concept, we focus on modeling wildfire evacuations using a built-in basic fire prediction model. In doing so, we exhibit the type and format of information required by our algorithm that real-world prediction models would need to provide in order to be incorporated. 

In Section \ref{sec:dynamicflows}, we saw that the key steps when converting a dynamic network to a time-expanded network are the discretization of time and the copying of nodes and arcs. We note that the number of nodes in a time-expanded network scales directly with the granularity of the discretized time horizon. In our proof of concept computations in Section \ref{sec:test_cases}, we typically observe time horizons of 20-50 minutes to complete a regional evacuation, and use a minute as a unit of time to balance model detail and efficiency of solution. Travel times along arcs are rounded up to the nearest integer.

Depending on the spread of a wildfire, nodes in the time-expanded network may not be visited at certain time instances, and arcs may have reduced capacity or may not be used. Therefore, to accurately model hazard information in the time-expanded network, we cannot simply copy our node and arc sets for each time step. Instead, we must incorporate additional information, i.e., the time at which each node and arc is impacted by the hazard. There exist various prediction models for different hazards, such as wildfires and floods.  Our algorithm can incorporate any model that, for each time instance $t$ in the discretized time horizon, provides the area or areas that are impacted by the hazard; we denote this set $F(t)$. The necessary notation for the underlying TEN has been introduced in Section \ref{sec:prelim}. Here we add notation regarding the fire prediction model used in conjunction with TEN, to yield a new network WTEN, the {\em wildfire time-expanded network}.  Note that we use this name for ease, in reality, ``wildfire" may be replaced with any appropriate hazard.

The sets $F(t)$ provide information that we incorporate into the network in two ways:  (1), nodes that are contained in $F(t)$, along with the arcs incident to them, are removed, and (2), capacities of arcs close to $F(t)$ are reduced (possibly to zero). Since we use Geographic Information System (GIS) data and a real-world road network, each node $i \in N$ has an associated latitude and longitude which we denote $(\lat_{i},\lon_{i})$. For each time instance, $t$, we formally describe the set of nodes $N_{F(t)}$ as the intersection of $N$ with $F(t)$,
\begin{equation}
    N_{F(t)} := \left\{i \in N: \left(\lat_{i},\lon_{i}\right) \in F(t) \right\},
\end{equation}
and represent the set of nodes that remain in the network for {\it every} time instance as
\begin{equation}
    N_W := \left\{i(t) \in N_T: i \in N \setminus N_{F(t)},\ t \in [T]_0 \right\}.
\end{equation}

Given these updates, we must revise the definitions of the set of source nodes, $S_T$, and the set of sink nodes, $D_T$. We define the source nodes of a WTEN to be the source nodes from the original network which have not been overtaken by the fire already at $t=0$. The sink nodes are defined to be the final copy of each of the original sink nodes in the WTEN. Note that we assume a node that has been overtaken by the fire at some time instance can not become available again during a later time instance.  More precisely, for every time instance, $t$, we assume
$$F(t) \subseteq F(t+1).$$

For a given sink $d_j$, we define $t_{\max}(j)$ to be the largest time interval such that $d_j$ does not intersect with $F(t)$. Formally, this gives the following definitions:
\begin{equation}
    t_\text{max}(j) := \max\{t \in [T]_0: d_j \notin N_{F(t)}\}
\end{equation}
\begin{equation}
    S_W:=\{s_{i}(0) \in N_W: s_{i}(0) \in S_T \setminus N_{F(0)}\} \;\; \text{and}
\end{equation}
\begin{equation}
D_W:=\{d_{j}(t_\text{max}(j))\in N_W: d_j \in D\}
\end{equation}

Finally, we must update the movement and holdover arc sets accordingly, to account for $N_W$, $S_W$, and $D_W$:
\begin{align}
    A^M_W &:= \{(i(t),j(t')) \in A^M: (i,j) \in A,\  i \in N \setminus N_{F(t)},\ j \in N \setminus N_{F(t')},\notag \\
    & \qquad \qquad \qquad \qquad \qquad \qquad \qquad \qquad \qquad \qquad t' = t+\lambda_{ij}(t) \leq T, \ t \in [T-1]_0\}\\
    A^H_W &:= \{(i(t),i(t+1)) \in A^H: i \in (S_W \cup D_W) \setminus (N_{F(t)} \cup N_{F(t+1)}),\ t \in [T-1]_0\},
\end{align}

To see an example, recall the TEN and corresponding dynamic network displayed in Figure \ref{fig:time_expanded_net}.  In Figure \ref{fig:wten_example}, we show the same TEN alongside a WTEN constructed from the same dynamic network but exhibiting a situation in which source node $1 \in N$ has been overtaken by the fire at $t=2$. Note that node $2 \in N$ at $t=3$ is not overtaken, so its copy $11 \in N_W$ exists in the network. However, it is disconnected due to $1 \in N$ being overtaken.

\begin{figure}[htb!]
\centering
\begin{minipage}{0.75\textwidth}
\centering
     \begin{subfigure}[a]{\textwidth}
     \centering
         \subcaptionbox{Time-Expanded Network for the dynamic network in Figure \ref{fig:time_expanded_net}. \label{fig:time_exp_net}}{
\begin{tikzpicture}[scale = 0.85]
\tikzset{vertex/.style = {shape=circle,draw,minimum size=2em}}
\node[vertex] (1) at (0,2) {1};
\node[vertex] (1') at (3,2) {4};
\node[vertex] (1'') at (6,2) {7};
\node[vertex] (1''') at (9,2) {10};
\node[vertex] (2) at (0,0) {2};
\node[vertex] (2') at (3,0) {5};
\node[vertex] (2'') at (6,0) {8};
\node[vertex] (2''') at (9,0) {11};
\node[vertex] (3) at (0,-2) {3};
\node[vertex] (3') at (3,-2) {6};
\node[vertex] (3'') at (6,-2) {9};
\node[vertex] (3''') at (9,-2) {12};
\draw[-latex,thick] (3) to node[midway, fill=white] {$a_3(0)$} (3');
\draw[-latex,thick] (3') to node[midway, fill=white] {$a_3(1)$} (3'');
\draw[-latex,thick] (3'') to node[midway, fill=white] {$a_3$(2)} (3''');
\draw[-latex,thick] (1) to node[midway, fill=white] {$a_1(0)$} (1');
\draw[-latex,thick] (1') to node[midway, fill=white] {$a_1(1)$} (1'');
\draw[-latex,thick] (1'') to node[midway, fill=white] {$a_1$(2)} (1''');
\draw[-latex,thick] (1) to node[midway, fill=white] {$u_{12}(0)$} (2');
\draw[-latex,thick] (1') to node[midway, fill=white] {$u_{12}(1)$} (2'');
\draw[-latex,thick] (1'') to node[midway, fill=white] {$u_{12}(2)$} (2''');
\draw[-latex,thick] (1) to node[midway, fill=white] {$u_{13}(0)$} (3');
\draw[-latex,thick] (1') to node[midway, fill=white] {$u_{13}(1)$} (3'');
\draw[-latex,thick] (1'') to node[midway, fill=white] {$u_{13}(2)$} (3''');
\draw[-latex,thick] (2) to node[midway, fill=white, xshift=2mm] {$u_{23}(0)$} (3'');
\draw[-latex,thick] (2') to node[midway, fill=white, xshift = 2.5mm] {$u_{23}(1)$} (3''');
\end{tikzpicture}}
 \label{fig:wten_ex_a}
     \end{subfigure}
    \end{minipage}
    
\vspace{1em}
    \begin{minipage}{0.75\textwidth}
    \centering
     \begin{subfigure}[b]{\textwidth}
     \centering
         \hfill
       \begin{tikzpicture}[scale = 0.85]
\tikzset{vertex/.style = {shape=circle,draw,minimum size=2em}}
\node[vertex] (1) at (0,2.5) {1};
\node[vertex] (1') at (4,2.5) {4};
\node[vertex] (2) at (0,0) {2};
\node[vertex] (2') at (4,0) {5};
\node[vertex] (2'') at (8,0) {8};
\node[vertex] (2''') at (12,0) {11};
\node[vertex] (3) at (0,-2.5) {3};
\node[vertex] (3') at (4,-2.5) {6};
\node[vertex] (3'') at (8,-2.5) {9};
\node[vertex] (3''') at (12,-2.5) {12};
\draw[-latex,thick] (3) to node[midway, fill=white] {$a_3(0)$} (3');
\draw[-latex,thick] (3') to node[midway, fill=white] {$a_3(0)$} (3'');
\draw[-latex,thick] (3'') to node[midway, fill=white] {$a_3(0)$} (3''');
\draw[-latex,thick] (1) to node[midway, fill=white] {$a_1(0)$} (1');
\draw[-latex,thick] (1) to node[midway, fill=white, xshift = 3.5mm, yshift =0mm] {{$u_{12}(0)\cdot p_{12}(0)$}} (2');
\draw[-latex,thick] (1') to node[midway, fill=white, xshift = 3.5mm, yshift =0mm] {$u_{12}(0)\cdot p_{12}(1)$} (2'');
\draw[-latex,thick] (1) to node[midway, fill=white] {$u_{13}(0)\cdot p_{13}(0)$} (3');
\draw[-latex,thick] (1') to node[midway, fill=white] {$u_{13}(0)\cdot p_{13}(1)$} (3'');
\draw[-latex,thick] (2) to node[midway, fill=white, xshift = 6.15mm, yshift =0mm] {$u_{23}(0)\cdot p_{23}(0)$} (3'');
\draw[-latex,thick] (2') to node[midway, fill=white, xshift = 6.15mm, yshift =0mm] {$u_{23}(0)\cdot p_{23}(1)$} (3''');
\end{tikzpicture} 
        \caption{Example wildfire-expanded network (WTEN): nodes that do not appear for later time instances have been overtaken by the fire.}
        \end{subfigure}
        \label{fig:wten_ex_b}
        \end{minipage}
        \caption{A time-expanded network and associated wildfire-expanded network (source node $1 \in N$ has been overtaken by the fire at $t=2$).}
        \label{fig:wten_example}  
\end{figure}

Next, we adjust arc capacities for the movement arcs remaining in the WTEN. Specifically, we assume that the distance from a road to the hazard directly affects how many people are able to safely travel along the road in a given time window. This is intuitive, for example, in a fire, smoke may reduce visibility on the road, leading to drivers slowing down. The impact can be interpreted in two ways, either of which could be directly incorporated into our model: drivers take longer to traverse that stretch of road; or fewer people are able to make it from the start to the end of the road during the same travel time. One can model this effect by increasing travel times $\lambda_{ij}(t)$ or reducing capacities $u_{ij}(t)$ over time, respectively. We elect to model it by reducing the upper capacities $u_{ij}(t)$ of movement arcs close to the fire, and keep the travel times $\lambda_{ij}(t)$ along an arc the same. This implies that the updated $A^M_W$ remains fixed, and only capacities need to be adjusted. In our fire prediction model, we also assume the fire is growing at a rate of one meter per minute \citep{a-76}. This allows us to directly compare travel times to the distance to the fire without needing to convert units. We define the parameter $p_{ij}(t)$ as the percentage of available capacity of arc $(i,j) \in A$ at time $t$:
\begin{equation}
    p_{ij}(t) =
\begin{cases}
    1 & f_{ij}(t) \geq \lambda_{ij}(t)\\
    f_{ij}(t)/\lambda_{ij}(t) & f_{ij}(t) < \lambda_{ij}(t),
\end{cases}
\end{equation}
where $f_{ij}(t)$ is the distance from arc $(i,j) \in A$ to $F(t)$ at time $t$. Here $p_{ij}(t)$ expresses the percentage of the original capacity $u_{ij}(0)$ that remains at time period $t$. We assume that the time instance $t$ in $p_{ij}(t)$ corresponds to $(i(t),j(t')) \in A_W$ where $t' = t + \lambda_{ij}(t)$. If $p_{ij}(t)$ falls below 20\% we set the capacity for that given arc to 0. In this case, when the fire is farther than or equal to the travel time $\lambda_{ij}(t)$, the arc is at full capacity $u_{ij}(0)$. Otherwise, the capacity reflects the encroachment of the fire and allows fewer people to travel along the arc. Note that a computation of $f_{ij}(t)$ is quite challenging; we provide more detail in Section \ref{sec:compeff}.

Using these definitions, we calculate $u^W_{i(t)j(t')}$ for movement arcs as
\begin{equation}
    u^W_{i(t)j(t')} = u_{ij}(t) = \lfloor u_{ij}(0) \cdot p_{ij}(t) \rfloor,
\end{equation}
where $i(t),j(t') \in N_W$ and $t'=t+\lambda_{ij}(t)$ are some copy of $i,j \in N$ in the WTEN such that $(i(t),j(t')) \in A^M_W$. In our model, holdover arcs are treated as having the same capacity for all time instances, i.e., $a_i(0)=a_i(t) = a_i$ for all $1\leq t \leq T$ and for all $i \in (S_W \cup D_W)$. However, it would be simple to adjust these capacities based on the distance $f_{ii}(t)$, i.e, the distance between $i \in N$ and $F(t)$, if desired. For example, those source nodes closest to the fire could have reduced holdover capacity between earlier time instances to force evacuees to leave quickly. This type of capacity reduction can also replicate some ideas of a stage-based evacuation in which specific nodes, such as those with the largest at-risk populations, begin evacuating early.

Finally, to allow for multiple sources and sinks for our evacuation problem, we introduce a super source, $\hat s$, and super sink, $\hat d$, to solve a maximum flow problem.  To do so, one adds arcs from $\hat s$ to the sources $s_{i}(0)$ and from sinks $d_{j}(t_\text{max}(j))$ to $\hat d$ of the WTEN, respectively. This gives an arc set of {\it super arcs}
\begin{equation}
    A^{(S_W \cup D_W)}_W := \left\{\left(\hat s,j\right):j \in S_W\right\} \cup \{(i,\hat d): i \in D_W \}.
\end{equation}

There is another benefit to introducing these artificial arcs; the ability to specify upper bounds for specific evacuation sources. In an evacuation, the number of people who use a specific source node as the starting point is limited. Similarly, sink nodes may have a maximum capacity since they correspond to locations such as arenas or schools. Because of this, arcs in $A^{(S_W \cup D_W)}_W$ are set to have an upper capacity based on the holdover capacity $a_i(0)$ for $i \in (S_W \cup D_W)$.

We now summarize the construction of the WTEN via a formal definition.
\begin{definition}[Wildfire Time-Expanded Network]\label{def:wten}
The wildfire time-expanded network (WTNE) $G_W = (N_W,A_W)$ of a dynamic network $G = (N,A,T)$ over a time horizon $T$ with fire sets $F(t)$ for all $t \in [T]_0$ is defined by

\begin{equation}
    N_W := \left\{i(t) \in N_T: i \in N \setminus N_{F(t)},\ t \in [T]_0 \right\} \cup \{\hat s, \hat d\} \text{ and } A_W := A^{M}_W \cup A^{H}_W \cup A^{(S_W \cup D_W)}_W
\end{equation}
where
\begin{align*}
    A^M_W &:= \{(i(t),j(t')) \in A^M: (i,j) \in A,\  i \in N \setminus N_{F(t)},\ j \in N \setminus N_{F(t')},\\
    & \qquad \qquad \qquad \qquad \qquad \qquad \qquad \qquad \qquad \qquad t' = t+\lambda_{ij}(t) \leq T, t \in [T-1]_0\}\\
    A^H_W &:= \{(i(t),i(t+1)) \in A^H: i \in (S_W \cup D_W) \setminus (N_{F(t)} \cup N_{F(t+1)}),\ t \in [T-1]_0\}\\
    A^{(S_W \cup D_W)}_W &:= \left\{\left(\hat s,j\right): j \in S_W\right\} \cup \{(i,\hat d): i \in D_W\},
\end{align*}

for the movement, holdover, and super arcs respectively. The capacities $u^W_{i(t)j(t')}$, $u^W_{i(t)i(t+1)}$, $u^W_{\hat s j'}$, and $u^W_{i' \hat d}$ for movement, holdover, and super arcs respectively are set as
\begin{align}
   && u^W_{i(t)j(t')} &:= \lfloor u_{ij}(0)p_{ij}(t) \rfloor &&\text{for } i,j \in A, \ t' = \lambda_{ij}(t) + t, \ t \in [T-1]_0&&\\
    && u^W_{i(t)i(t+1)} &:= a_i(0) &&\text{for } i \in (S \cup D) , \ t \in [T-1]_0&&\\
    && u^W_{\hat s j'} &:= a_{s_j}(0) &&\text{for } j' \in S_W,\ s_j \in S, \ j' = s_j(0)&&\\
    && u^W_{i' \hat d} &:= a_{d_i}(0) &&\text{for } i' \in D_W, \ d_i \in D,\ i' = d_i(t_\text{max} (i)).&&
\end{align}
\end{definition}

The maximum flow problem on WTEN differs from (\ref{TEN_LP}) in its restriction to a subnetwork and adjusted capacities, and the addition of a super source and sink with arcs of bounded capacity. We formally write the corresponding LP as:
\begin{align}
    \text{max } && \displaystyle{\sum_{(i,\hat d) \in A_{W}}} x_{i{\hat d}} && \\
    \text{s. t. } && \displaystyle{\sum_{(j,i)\in A_{W}}}x_{ji} - \displaystyle{\sum_{(i,j) \in A_{W}}}x_{ij} = 0 &\quad& \forall i \in N_{W} \setminus  \{\hat s, \hat d\}\\
    &&  0 \leq x_{ij} \leq u^W_{ij} &\quad& \forall (i,j) \in A_{W}.\\
    \tag{WTEN LP}
    \label{WTEN_LP}
\end{align}

\subsection{Wildfire Modeling}

Traditional wildfire models can be computationally expensive or require expert knowledge to use. Moreover, with the availability of more and better data, wildfire models have grown in sophistication, detail, and accuracy. Today, many models take into account weather conditions, amount and types of burnable materials, and even flame height and length \citep{ofbsevm-23,pp-11}. These complex models may not be readily available for all levels of evacuation and response teams, especially those working on a local scale. To accommodate this potential barrier, we provide a routine in our repository allowing the user to model hazard (in particular, fire) growth and impact area over time. This enables anyone with access to our code to obtain clear, implementable evacuation plans based on approximate hazard development, while also allowing experts to easily integrate information obtained from sophisticated wildfire models should they have access to such data. 

More precisely, the code for our algorithm allows for the incorporation of any wildfire model via a GIS shapefile output that represents where the fire is, i.e., a fire set $F(t)$, at a given time, $t$. That is, any geometry or shapefile that is compatible with the Python Shapely package can be integrated with our algorithm. A detailed, state-of-the-art fire model can be used to generate the shapefile for an initial plan, where computational speed is not a bottleneck. Those plans can also be pre-computed and used as a template. However, during a developing fire and evacuation, computational speed is vital to enable a rapid reaction to a change in fire spread. In such a situation, one would need to use a simple fire model to avoid a bottleneck.
 
To create a simple, but reasonable model of fire spread, we tested the effect of defining $F(t)$ to be various 2-dimensional sets, such as unions of squares or circles.  We observed through computation that the precise nature of these sets does not strongly impact the practical efficiency of the intersection and distance calculations required to integrate wildfire information into the network. 
For our sample runs in Section \ref{sec:test_cases}, we model the fire as a union of circles whose radii grow linearly over time. This choice is only reflected in the definition of the fire set below; the modeling steps are in all generality.

We define $F(t)$ at time $t \in [T]_0$ as
\begin{equation}
    F(t):=\left\{(x,y) \in \R^2: (x,y) \in \bigcup_{k=1}^{\ell} C_k\right\},
\end{equation}
where $C_k$, $k \in \{1,\dots,\ell\}$, is a circle centered at $(x_k,y_k)$ with radius $r_k(t)$ for time instance $t$. The fire is modeled as the union of $\ell$ circles; the radius of each circle is time-dependent and grows to simulate the fire growing over time.

\subsection{Efficiency of WTEN Construction}
\label{sec:compeff}

We conclude this section with a closer look at the efficiency of construction of the WTEN. To this end, let us first provide some information on the computation of distances $f_{ij}(t)$ between movement arcs and the fire sets. We use the Shapely Python package, which has built-in functions to calculate the distance between various geometric objects. In particular, it can calculate the distance between LineStrings, edges in the road network, and MultiPolygons, the fire set $F(t)$ at a given time $t \in[T]_0$.
LineStrings can contain multiple points which are used to map the single road segment to geographic coordinates; see $(1,2)$ and $(1,3)$ in Figure \ref{fig:dist_ex}. When a road is curved, one may require multiple lines to represent its true shape. While it would be possible to approximate the distance from an edge to a fire set based on a straight line between its nodes, we found that this can result in a significant loss of information in practice.

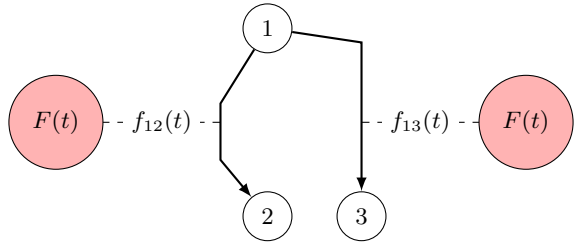
\begin{figure}[h!]
    \centering
    \begin{tikzpicture}[scale = 1.25]
    \tikzset{vertex/.style = {shape=circle,draw,minimum size=2em}}
    \node[vertex] (1) at (0,2) {1};
    \node[vertex] (2) at (0,0) {2};
    \node[vertex] (3) at (1,0) {3};
    \draw[-latex,thick] (1) to  (-0.5,1.2) to (-0.5,0.60) to (2);
    \draw[-latex,thick] (1) to (1,1.85) to (3);

    \draw[fill = red!30] (-2.25,1) circle [radius=0.5] node {$F(t)$};
    \draw[fill = red!30] (2.75,1) circle [radius=0.5] node {$F(t)$};

     \draw[dashed] (-0.5,1) to node[midway, fill=white] {$f_{12}(t)$} (-1.75,1);
     \draw[dashed](1,1) to node[midway, fill=white] {$f_{13}(t)$} (2.25,1);
    \end{tikzpicture}

    \caption{Distances between two LineStrings  $(1,2)$, $(1,3)$ (roads) and MultiPolygon $F(t)$ (fire). Distances $f_{ij}(t)$ for the arcs are based on the closest components of the LineString and MultiPolygon.}
    \label{fig:dist_ex}
\end{figure}
The curvature of edge $(1,2)$ in Figure \ref{fig:dist_ex} provides an example where using the distance of $F(t)$ from the straight line between nodes 1 and 2 would yield a larger capacity than what edge $(1,2)$ could safely handle due to its short distance to the fire set $F(t)$ in the middle segment.The edge $(1,3)$ provides a similar example. The figure also displays a situation in which the use of a MultiPolygon for the fire set $F(t)$ is necessary to retain viable information. Not only does the use of a union of circles (instead of a single circle or polygon) allow the representation of more complicated fire area shapes, but it also facilitates the representation of disconnected fire areas, and one computes the smallest distance to any of them.

We introduce some notation for a formal running time statement for the construction of the WTEN. Let $f_N$ represent the maximal cost of checking whether a node $i$ satisfies $i \in N_F(t)$ and let $f_A$ represent the maximal cost of computing $f_{ij}(t)$ for an arc $(i,j) \in A$ and $t \in [T]_0$.

\begin{lemma}\label{lem:setupWTEN}
    Let (\ref{Dynamic}) represent a dynamic flow problem on a dynamic network $G=(N,A,T)$
    and let $F(t)$ be a fire set for $t\in [T]_0$. The associated WTEN $G_W = (N_W,A_W)$ satisfies $|N_W|\leq (T+1)|N|+2$ and $|A_W|\leq T(|A|+|N|)+|N|$,
   and can be constructed in time $O(T|N|\cdot f_N + T|A|\cdot f_A)$ in the arithmetic model of computation.
\end{lemma}

\begin{proof}
By Proposition \ref{prop:setup}, a TEN $G_T=(N_T,A_T)$ associated to (\ref{Dynamic}) can be constructed in linear time $O(T|A|)$. The WTEN $G_W$ is then built as a subnetwork of $G_T$, which has sizes $|N_T|= (T+1)|N|$ and $|A_W|\leq T(|A|+|N|)$, and by adding two nodes, super source and sink, as well as at most $N$ super arcs, one for each node in $S_W\cup T_W$. This confirms the size of the network. It remains to prove the running time of the construction.

First, we consider the construction of the node sets. The super source and sink nodes are added at constant time. Construction of $N_{F(t)}$ and $N_W$ requires a single check $i \in N_F(t)$ for each node $i\in N$ and $t\in [T]_0$, at cost $f_N$ each, to decide whether to either add $i(t)$ to $N_W$ or $i$ to $N_{F(t)}$. The set $S_W$ can be constructed with constant additional effort for each node. The set $D_W$ requires the identification of $t_{\max}(j)$ for each $d_j$, which can be done in time $O(|T|)$ for each node $d_j\in D$. Together the node sets $N_{W}$, $S_W$, and $D_W$ can be constructed in time $O(T|N|\cdot f_N+T|N|)\subset O(T|N|\cdot f_N)$.

Next, we consider the construction of the arcs sets. The super arcs $A^{(S_W \cup D_W)}$ are added in constant time each, for a total of $O(|S_W \cup D_W|)\subset O(N)$. The holdover arcs $A_W^H$ are selected from the holdover arcs $(i(t),i(t+1)) \in A^H$ by verifying membership of the respective $i(t), i(t+1)$ in $N_W$. This is constant effort for each arc in $A^H$, and gives effort $O(T|N|)$. Similarly, the movement arcs $A_W^M$ are selected from the movement arcs $(i(t),j(t')) \in A^M$ by verifying membership of the respective $i(t), j(t')$ in $N_W$, giving effort $O(T|A|)$. Thus, the arc sets are constructed in $O(T|A|)$ overall.

Finally, for each  arc $(i,j) \in A$ and each $t \in [T]_0$, computation of $f_{ij}(t)$ takes time $f_A$. It then takes constant effort to devise $p_{ij}(t)$ from it and to set capacity $u^W_{ij}(t)$ as $u^W_{ij}(t)=p_{ij}(t) \cdot u_{ij}(t)$, or $0$ if $p_{ij}(t) \leq 0.2$. This gives $O(T|A|\cdot f_A)$.

Overall, the running time is $O(T|N|\cdot f_N + T|A|+ T|A|\cdot f_A)\subset O(T|N|\cdot f_N + T|A|\cdot f_A)$. \qed
\end{proof}

Generally, $f_A$ dominates $f_N$: it is at least as hard to find the distance $f_{ij}$ between two general geometric objects as it is to find the distance between a point and an object. Thus, the running time simplifies to $O(T|A|\cdot f_A)$. For example, consider our fire set $F(t)$, which is a union of circles. Here, $f_N$ is based on the computation of the distance between a node and the center of each of the circles, and checking whether any of these distances is lower than the corresponding circle's radius.  In contrast, $f_A$ is based on an orthogonal projection of the center of each circle onto each of the lines (line segments) in the LineString representing a road, and taking a minimum. For a longer curved road, we observed up to 1,414 such lines segments for a single LineString in our data. These computations form the largest computational expense in our algorithm; see Tables \ref{tab:sum_init_evac} and \ref{tab:sum_update} in Section \ref{sec:runningtime}.

\section{The Algorithm}\label{sec:alg}

Our algorithm consists of three main parts: the pre-processing of public data (Section \ref{sec:data_processing}), the construction of an initial evacuation plan (Section \ref{sec:init_soln}), and the creation of updated plans based on real-time fire data (Section \ref{sec:plan_update}). For each part, we provide a high-level overview of the functions and sub-routines. We also provide summary flowcharts for the initial and updated evacuation plans in Figure \ref{fig:flow_chart} and Figure \ref{fig:flow_chart_update}, respectively.

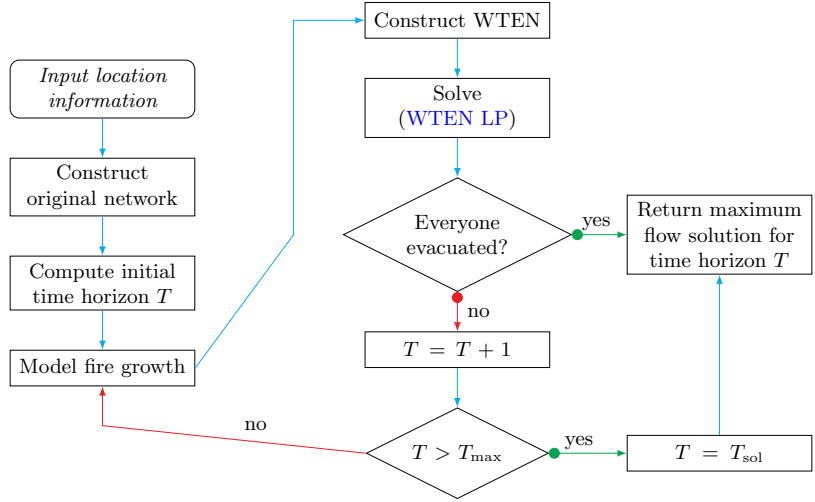
\begin{figure}[htb]
\centering
\resizebox{0.75\textwidth}{!}{
\begin{tikzpicture}[
    >=triangle 60,
    start chain=going below, 
    node distance=6mm and 40mm, 
    every join/.style={norm},
    ]
\tikzset{
  base/.style={draw, on chain, on grid, align=center, minimum height=4ex},
  proc/.style={base, rectangle, text width=8em},
  test/.style={base, diamond, aspect=2, text width=5em},
  term/.style={proc, rounded corners},
  coord/.style={coordinate, on chain, on grid, node distance=6mm and 25mm},
  nmark/.style={draw, cyan, circle, font={\sffamily\bfseries}},
  norm/.style={-latex, ultra thick, draw, Cyan},
  it/.style={font={\small\itshape}}
}
\node[coord] (c3){};
\node [term, it] (1) {Input location information};
\node[proc, join,] (2) {Construct original network};
\node[proc, join] (3) {Compute initial time horizon $T$};
\node[proc, join] (p0) {Model fire growth};
\node[coord] (c1) {};
\node[coord] (c5) {};
\node[proc, right=of c3] (p1) {Construct WTEN};
\node[proc, join] {Solve (\ref{WTEN_LP})};
\node[test, join] (t4) {Everyone evacuated?};
\node[proc] (p4) {$T = T+1$};
\node[test, join] (t5) {$T > T_\text{max}$};
\node[proc, right=of t4] (p5) {Return maximum flow solution for time horizon $T$};
\node[proc, right=of t5] (p6) {$T = T_\text{sol}$};
\node[coord, left=of p1] (c2) {};
\node[coord, left=of t4] (c4) {};
\path (t4.south) to node [near start, xshift=1em, yshift = -0.5em] {no} (p4.north);
  \draw [-latex, ultra thick, Red] (t4.south) -- (p4.north);
\path (t5.west) to node [near start, xshift = -2.75em, yshift=0.5em] {no} (p0);
  \draw [-latex,ultra thick, Red] (t5.west) -- (c1)  -| (p0.south);
\path (t4.east) to node [near start,xshift=0.25em, yshift= 1em] {yes} (p5.west);
\draw [-latex,ultra thick, Green] (t4.east) -- (p5.west);
\path (t5.east) to node [near start,xshift=0.5em, yshift= 1em] {yes} (p6.west);
  \draw [-latex,ultra thick, Green] (t5.east) -- (p6.west);
  \draw [-latex, ultra thick, Cyan] (p0.east) -- (c4)  -- (c2) -- (p1.west);
\draw [-latex,ultra thick, Cyan] (p6.north) -- (p5.south);
\end{tikzpicture}
}
\caption{Flowchart of the steps to construct an initial evacuation plan.}
\label{fig:flow_chart}
\end{figure}

\subsection{Data and Pre-processing}\label{sec:data_processing}

OpenStreetMaps is an open source GIS database of geographic and road information \citep{openstreetmap}. It can produce road data for geographic locations based on various inputs such as bounding boxes, a select distance from a center point, or a city/county name. For more information regarding OpenStreetMaps data and license and OMSNX package we refer the reader to \citep{OSMNX-17,openstreetmap}. Since OpenStreetMaps is compiled by the public, road data may not be available for all input locations, however, a major benefit of open source data is that it is updated regularly and therefore can lead to greater accuracy. 

Once a geographic location has been stated, we use the Python package OSMNX to import the road network and construct a network object. This network object contains information for the arcs (roads), as well as nodes for intersections and dead ends.
The density of the network can be adjusted by contracting nodes based on their distance (in meters) from neighboring nodes, called \emph{tolerance}. This is used to contract redundant intersections, such as the multiple exits of roundabouts. The specification of a higher tolerance also facilitates work on larger road networks, whole counties instead of cities;
see Figure \ref{fig:tol_diff} for an example.  The cities and counties used for our test cases explore networks of varying sizes.

\begin{figure}[htb]
     \centering
     \begin{subfigure}[b]{0.475\textwidth}
         \centering
         \subcaptionbox{Butte County with a Tolerance of 10 meters\label{fig:butte_tol_10}}{
         \includegraphics[scale = 0.4]{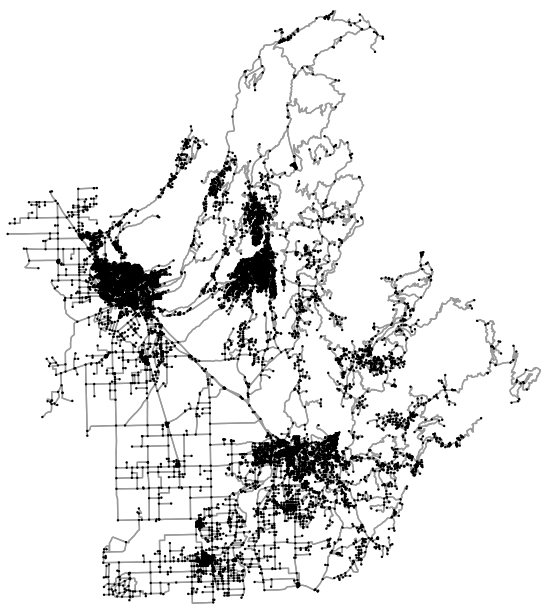}
         }
     \end{subfigure}
     \hfill
     \begin{subfigure}[b]{0.475\textwidth}
         \centering
         \subcaptionbox{Butte County with a Tolerance of 100 meters\label{fig:butte_tol_100}}{
         \includegraphics[scale = 0.4]{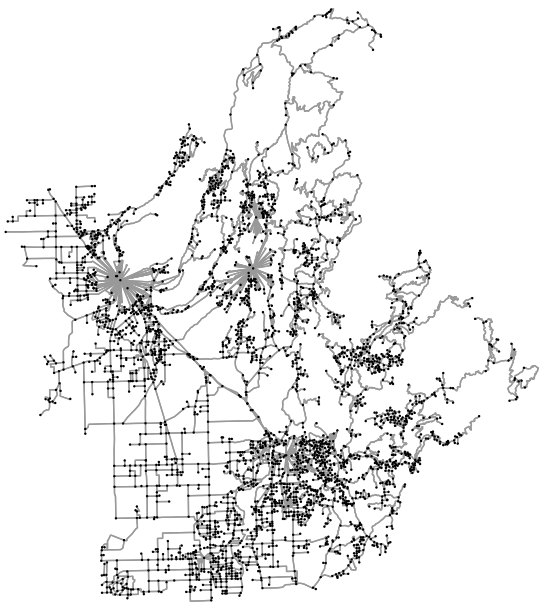}
         }
     \end{subfigure}
        \caption{Butte County, California road network with tolerances of 10 meters (left) and 100 meters (right). Nodes of the network are shown in black; edges are gray.}
        \label{fig:tol_diff}
\end{figure}

 The network object is updated to include the parameters for a dynamic network; see Section \ref{sec:prelim}. This information can be appended to the network object using built-in OMSNX functions. Travel times can be added directly from OpenStreetMaps (which are rounded to integers during the TEN creation), while the computation of road capacities requires additional steps. First, the speed for each road is added to the network. The capacity for each road is then computed using the Moore Method \citep{mkp-13}, a simple calculation using empirical data on safe vehicle distance; it is rounded down and attached to the corresponding arc. 
 Finally, each node may be given a supply or demand value, determined by user input. A supply value identifies a source, where the corresponding number of people start their evacuation. A demand value identifies a sink; a destination location with the corresponding capacity.

Next, we establish an initial time horizon $T$. We are looking for a minimal $T$ that will allow for a complete evacuation.  A lower bound on this $T$ is the longest of the shortest paths from any source to its closest sink in the original network (based on travel time). There are $|S|\cdot |D| \leq \left(|N|/2\right)^2$ such paths to compute, each of them requiring effort $O((|A|+ |N|) \log |N|)$ via Dijkstra's Algorithm; see for example \citep{amo-93}. We find this value using a built-in NetworkX function which computes the shortest path using original travel times between any two nodes in the dynamic network.  

We then go through the path and determine how many segments would be rounded during the conversion to the WTEN.

The fire is then modeled for all times $t \in [T]_0$. For each time instance, $t \in [T]_0$, a fire set, $F(t)$, a list of nodes overtaken by the fire, $N_F(t)$, and the distances from each road to the fire, $f_{ij}(t)$, are determined and stored. The pre-processing concludes with this step, displayed in the flowchart in Figure \ref{fig:flow_chart} on the bottom left.

\subsection{Construction of an Initial Evacuation Plan}\label{sec:init_soln}

The pre-processed information is used for the construction of a first WTEN. We begin with some checks for the possibility of a feasible solution. It may be possible that the predicted fire reaches all sinks (safe locations) before an initial time horizon $T$. When no sinks are reachable, not everyone can be evacuated and different safe locations should be chosen. Essentially, one can see this as a check on the quality of the input.

If some sink is reachable, we proceed to construct the WTEN; see Section \ref{HTEN}. The addition of a super source and super sink sets up a classical maximum flow problem. For the solution of (\ref{WTEN_LP}), we use NetworkX's built-in maximum flow algorithm. It allows the user to choose which algorithm they wish to use, such as Edmonds and Karp's maximum flow algorithm \citep{ek-72} or Goldberg and Tarjan's push relabel algorithm \citep{gt-88}. For a complete list of options we refer the reader to \citep{hss-08}. We use Dinic's shortest augmenting path algorithm \citep{din-70} to solve (\ref{WTEN_LP}). As we will see in Section \ref{sec:runningtime}, the solution of the problem is a small and efficient part of the overall running time of our algorithm.

We provide a formal running time statement for setup of the WTEN and maximum flow solution. Recall that $f_A$ represents the maximal cost of computing $f_{ij}(t)$ for an arc $(i,j) \in A$ and $t \in [T]_0$.

\begin{theorem}\label{thm:efficiency}
    Let (\ref{Dynamic}) represent a dynamic flow problem on a dynamic network $G=(N,A,T)$ 
    and let $F(t)$ be a fire set for $t\in [T]_0$. The associated (\ref{WTEN_LP}) can be set up and solved in time $O(T|A|\cdot (T|N| + f_A))$ in the arithmetic model of computation.
\end{theorem}

\begin{proof}
    Setup of the WTEN $G_W = (N_W,A_W)$ requires time $O(T|N|\cdot f_N + T|A|\cdot f_A)\subset O(T|A|\cdot f_A)$, by Proposition \ref{prop:setup} and Lemma \ref{lem:setupWTEN} (and paragraph after).

    Further, $G_W$ satisfies $|N_W|\leq (T+1)|N|+2 \in O(T|N|)$ and $|A_W|\leq T(|A|+|N|)+|N| \in O(T|A|)$, by Lemma \ref{lem:setupWTEN}. A running time of $O(|A_W|\cdot |N_W|) \subset O(T|A|\cdot T|N|)$ for the solution of a maximum flow problem is possible using the algorithm in \citep{o-13}. Together, this gives the claimed running time. \qed
\end{proof}

We note that the software package used contains simpler maximum flow algorithms than the one in \citep{o-13}. The options mentioned above have running times of $O(T|N|\cdot (T|A|)^2)$ for Edmonds-Karp, $O(T|N|^3)$ for Goldberg and Tarjan's algorithm, and $O((T|N|)^2\cdot T|A|)$ for Dinic's shortest augmenting path algorithm.

Finally, the algorithm checks if the maximum flow found is as large as the population that needs to evacuate. If the answer is yes, we return a dictionary with flow values for each arc in the WTEN and the time horizon $T$ that allowed for the complete evacuation. If the answer is no, we set $T=T+1$. 

We then check if this reaches a termination criterion value for $T$. In practice, we check if $T=T_{\max}$, where $T_\text{max}$ represents a worst-case time horizon for the algorithm set based on expert knowledge. If so, we return the maximum flow solution for time horizon $T_\text{sol}\leq T_{\max}$, which is the smallest time horizon where we first achieved the current maximum flow value. Otherwise, we increase the time horizon by one and rerun the algorithm with expanded fire data that accounts for the new time instance and a larger time-expanded network. Figure \ref{fig:flow_chart} displays this loop.

We verify correctness of this approach, in the sense of termination with a maximum flow at a minimal $T$. For this, we assume $F(t)\subset F(t+1)$ for all $t$. Informally, the fire can only increase in size over time, which in turn leads to non-increasing arc capacities over time. The following statement is for a formal version of the algorithm, where the initial time horizon is computed as described in Section \ref{sec:data_processing} and no hard upper bound $T_{\max}$ is specified a priori.

\begin{theorem}\label{thm:correctness}
Let (\ref{Dynamic}) represent a dynamic flow problem on a dynamic network $G=(N,A,T)$ with initial time horizon $T$ and let $F(t)$ be a fire set for $t\in [T]_0$ satisfying $F(t)\subset F(t+1)$. The construction of an initial evacuation plan, displayed in Figure \ref{fig:flow_chart}, terminates with a maximum flow, and minimal time horizon over which this flow is possible.
\end{theorem}

\begin{proof}
For a fixed time horizon $T$, by construction, a solution of (\ref{WTEN_LP}) is equivalent to solving the associated dynamic flow problem. In the WTEN, the cut $(\hat s, N_W\backslash \{\hat s\})$ (between super source and the rest of the network) has capacity $u_{\max}=\sum_{j\in S_W} u_{\hat sj}$ equivalent to the number of people to evacuate (a complete evacuation), and so a maximum flow value for the WTEN is bounded above by $u_{\max}$.

The initial time horizon $T$, chosen as the longest of the shortest paths between any source and their closest sink, is a lower bound on a minimal time horizon for a complete evacuation. If a complete evacuation is not possible within the current time horizon $T$, then $T$ is increased by one and the loop restarted. If a maximum flow value equivalent to a complete evacuation is achieved for the current time horizon, the algorithm terminates correctly.

It remains to argue that, if there exists a time horizon $T'$ that allows for a complete evacuation, that time horizon is reached. Additionally, if the input does not allow for a complete evacuation (regardless of an arbitrary extension of the time horizon), the algorithm needs to recognize this correctly and return the maximum flow (which then is not equivalent to the number of people to evacuate) and time horizon at which it is achieved. Both of these points are resolved by providing a correct termination criterion.

Let $T_0= \sum_{ij\in A} \lambda_{ij}$ be the sum of travel times for all arcs in the original dynamic network. Observe that the travel time for any path from any node to another in the original network is bounded above by $T_0$; in particular, this includes paths from a source to a sink. Thus, if any sink can be reached from some source, (\ref{WTEN_LP}) has a feasible solution for WTEN $G_W=(N_W,A_W)$ with time horizon $T_0$ that contains flow along a path in the original network.

The same observation holds for later time instances. Let $G'_W$ be a WTEN for some time horizon $T'$ with a maximum flow of value $f'$. It is trivial to check whether $f'$ corresponds to a complete evacuation. If not, consider the expanded WTEN $G^0_W$ for time horizon $T'+T_0$ that has the same maximum flow value $f_0=f'$. Note that $G'_W\subset G^0_W$ and that a flow value of $f'$ in $G^0_W$ can be achieved using only arcs in time instances $[T']_0$, by mirroring the flow in $G'_W$. The associated residual network for $G^0_W$ and this flow has no augmenting paths between sources and sinks anymore. Thus, there does not exist a path connecting a
 source node at time instance $T'$ or later to one of the sink nodes in $G^0_W$. Additionally, it is not possible to send additional flow leaving a source at an earlier time instance  $[T'-1]_0$, to a sink in $G^0_W$. This implies that people starting their evacuation at time instance $T'$ or later, or additional people that have started but not completed their evacuation at time instance $T'$, cannot reach a sink, regardless of a further expansion of the time horizon due to $F(t)\subset F(t+1)$. We obtain a termination criterion in the form of the maximum flow not increasing for $T_0$ time instances.

Recall that, by their definition, the $u_{ij}(t)$ (for which we multiply by $p_{ij}(t)$ and round down) are integral throughout. Thus, anytime the maximum flow value increases due to an extension of the time horizon, the value increases by at least $1$, until it reaches its final value, bounded above by $u_{\max}$. This gives finite termination.  \qed

\end{proof}

In practice, computations are much better behaved than the theoretical upper bound from this proof. In our experiments, we observed that it is worthwhile to start with a larger initial time horizon $T$ to reduce the number of iterations. We found a start with a $T$, where we add the travel time of the longest shortest path to the number of arcs in said path, rounded up, to work well. Often, the final time horizon for a complete evacuation does not exceed such an initial time horizon by more than a factor of two or three, whereas $T_0$ (from the proof of Theorem \ref{thm:correctness}) is already a vast overestimate. Further, a fixed maximal time horizon for a realistic evacuation $T_{\max}$ would typically be provided by a domain expert. 

In this subsection, we presented the procedure our algorithm uses to create an initial evacuation plan before a fire occurs.  This could be combined with  the use of one or more predictive fire models, enabling an expert to determine initial evacuation routes under a variety of events. In this way, our tool can be used to test and develop many different evacuation scenarios that can be shared with residents and practitioners ahead of time. The next part of the algorithm, and the next subsection, focuses on real-time or short notice updates.

\begin{figure}[htb]
\centering
\resizebox{0.97\textwidth}{!}{
\begin{tikzpicture}[
    >=triangle 60,
    start chain=going below, 
    node distance=6mm and 40mm, 
    every join/.style={norm},   
    ]
\tikzset{
  base/.style={draw, on chain, on grid, align=center, minimum height=4ex},
  proc/.style={base, rectangle, text width=8em},
  test/.style={base, diamond, aspect=2, text width=5em},
  term/.style={proc, rounded corners},
  coord/.style={coordinate, on chain, on grid, node distance=6mm and 25mm},
  nmark/.style={draw, cyan, circle, font={\sffamily\bfseries}},
  norm/.style={-latex, ultra thick, draw, Cyan},
  it/.style={font={\small\itshape}}
}
\node [term, it] (1) {Input new fire data and $T$ from initial plan};
\node[coord] (c6) {};
\node[proc, right=of 1] (2) {Construct new fire for $t_\text{fire} \leq t < T$};
\node[coord] (c8) {};
\node[proc, right=of 2] (3) (3){Combine with old fire from $t_\text{reopt} \leq t < t_\text{fire}$};
\node[test] (8) {Everyone evacuated?};
\node[proc] (10) {$T = T+1$};
\node[test, left=of 10] (11) {$T > T_\text{max}$};
\node[proc] (12) {$T = T_\text{sol}$};
\node[proc, right=of 8] (9) {Return maximum flow solution for time horizon $T$};
\node[proc, left=of 8] (7){Solve (\ref{WTEN_LP})};
\node[proc, left=of 7] (6) {Construct WTEN for $t_\text{reopt} \leq t \leq T$};
\node[coord, right=of 12] (c1) {};
\node[coord, right=of c1] (c2) {};
\node[coord, left=of 11] (c3) {};
\node[coord, left=of c3] (c4) {};
\node[coord, left=of 6] (c5) {};
\node[coord, left=of c6] (c7) {};
\node[coord, right=of 3] (c9) {};
\node[coord] (c10) {};
\node[coord] (c11) {};

  \draw [-latex, ultra thick, Cyan] (3.south) -- (6.north);
  \path (8.south) to node [near start, xshift=1em, yshift = -0.25em] {no} (10.north);
  \draw [-latex,ultra thick, Red] (8.south) -- (10.north);
\path (8.east) to node [near start, xshift=0.25em, yshift = 1em] {yes} (9.west);
  \draw [-latex, ultra thick, Green] (8.east) -- (9.west);
\path (11.south) to node [near start, xshift=1.25em, yshift = -0.25em] {yes} (12.north);
  \draw [-latex, ultra thick, Green] (11.south) -- (12.north);
\draw [-latex, ultra thick, Cyan] (1.east) -- (2.west);
\draw [-latex, ultra thick, Cyan] (2.east) -- (3.west);
\draw [-latex, ultra thick, Cyan] (6.east) -- (7.west);
\draw [-latex, ultra thick, Cyan] (7.east) -- (8.west);
\draw [-latex, ultra thick, Cyan] (10.west) -- (11.east);
\draw [-latex, ultra thick, Cyan] (12.east) -- (c2)  -| (9.south);
\path (11.west) to node [near start,xshift=-0.5em, yshift= 0.5em] {no} (c3);
\draw [-latex, ultra thick, Red] (11.west) -- (c4)  -| (c5) -- (c7) -- (c6)-- (2.south);
\end{tikzpicture}
}
\caption{Flowchart of the steps to update an initial evacuation plan.}
\label{fig:flow_chart_update}
\end{figure}
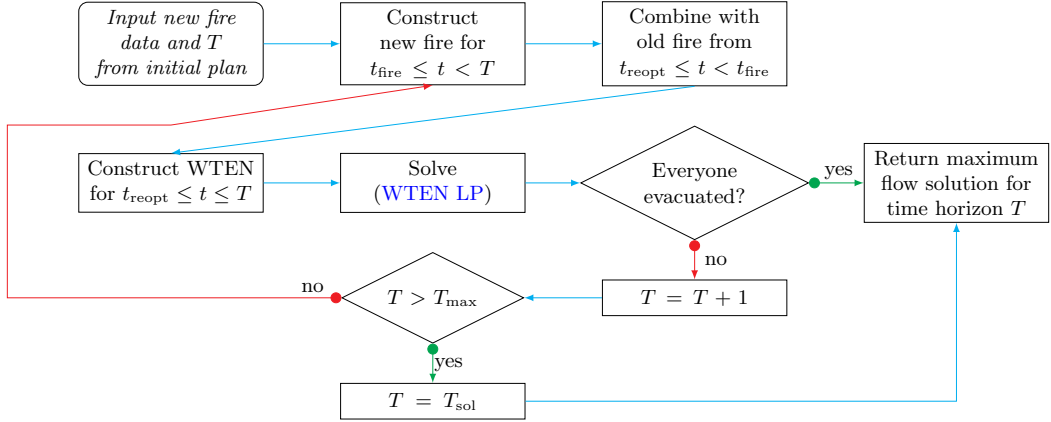

\subsection{Real-time Solution Updates} \label{sec:plan_update}

In the above, we described the setup of WTEN and the computation of an initial feasible solution to (\ref{WTEN_LP}) at minimal time horizon $T$ via a classical maximum flow algorithm. This feasible solution is the initial evacuation plan for an area.

We are interested in accounting for real-time changes in the fire data. For example, changing winds or ground conditions could cause the fire to grow faster than predicted or move to areas not captured in the original predictive model. To account for this, we dynamically update an existing evacuation plan on the WTEN. During this update, it is possible the time horizon must be increased in order to allow for a full evacuation. We account for this in the update in a similar manner to the construction of the initial evacuation plan.

Let us describe a possible timeline of events for such updates. During the early stages of an ongoing evacuation, experts may predict that the fire model needs to be revised (for example due to a faster spread) for time instances starting at
$t_\text{fire} \leq T$. Emergency crews are able to react to this prediction at time period $t_\text{reopt} < t_\text{fire}$, so a few time instances earlier. Thus, the computation of an update to the evacuation plan must be completed before $t_\text{reopt}$, and the evacuation plan can differ from the original plan starting at $t_\text{reopt}$ (even though the new fire information starts only at  $t_\text{fire}$).
For example, let us say that $10$ minutes into an ongoing evacuation fire officials see that the predicted model has changed and the fire is going to start getting significantly larger $15$ minutes later. They need $7$ minutes to relay this information to emergency crews. In this scenario, $t_\text{fire}=10+15=25$ and $t_\text{reopt} =10+7=17$.

To accomplish the update, we begin by creating an array of new fire data for all time instances $t_\text{fire} \leq t \leq T$. We combine it with the original predicted fire data for time instances $t_\text{reopt} \leq t < t_\text{fire}$. This gives us fire data for all time instances $t_\text{reopt} \leq t \leq T$. We are able to make adjustments to the initial evacuation plan starting at $t_\text{reopt}$. The ability to reoptimize the evacuation plan also for the `gap' between $t_\text{reopt}$ and $t_\text{fire}$ allows people to  change course in anticipation of the predicted fire change, and can increase the number of people that are safely evacuated compared with the initial evacuation plan.

Due to the start at $t_\text{reopt}$, we only need to construct a new WTEN for time instances $t_\text{reopt} \leq t \leq T$. We identify which nodes in any time instance $t_\text{reopt} \leq t \leq T$ in the initial evacuation plan had in-going flow from any time instance $0\leq t < t_\text{reopt}$. These nodes now become the new source nodes of the network and have a supply corresponding to that in-flow. We construct the WTEN using this updated information.

The shortest augmenting path algorithm is then used to solve (\ref{WTEN_LP}) for this smaller, new WTEN to obtain an updated evacuation plan. Figure \ref{fig:evac_init} and \ref{fig:evacu_update} show an example initial and updated evacuation plan. According to new fire information, nodes 7 and 11 are going to be overtaken and thus removed from the network. Arcs (1,4), (4,8), (4,9), (8,15), (9,12) and (12,15) are now dangerously close to the fire and have reduced capacity (and in turn flow). In this example, a maximum flow value towards sink node 15 decreases; the change in fire may lead to a scenario in which the evacuation is not completed within the same time horizon.

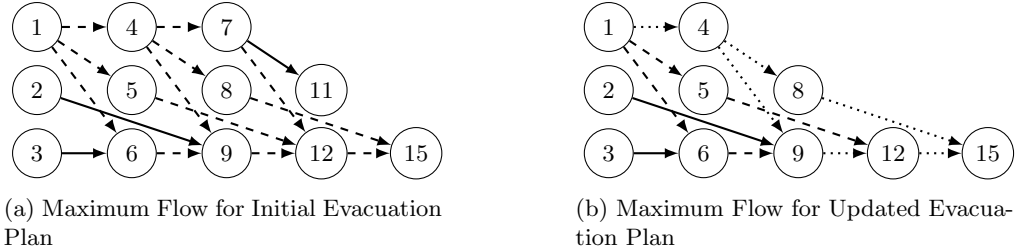
\begin{figure}[htb]
     \centering
     \begin{subfigure}[b]{0.475\textwidth}
         \centering
         \subcaptionbox{Maximum Flow for Initial Evacuation Plan \label{fig:evac_init}}{
\begin{tikzpicture}[scale = 0.42]
\tikzset{vertex/.style = {shape=circle,draw,minimum size=2em}}
\node[vertex] (1) at (0,2) {1};
\node[vertex] (1') at (3,2) {4};
\node[vertex] (1'') at (6,2) {7};
\node[vertex] (2) at (0,0) {2};
\node[vertex] (2') at (3,0) {5};
\node[vertex] (2'') at (6,0) {8};
\node[vertex] (2''') at (9,0) {11};
\node[vertex] (3) at (0,-2) {3};
\node[vertex] (3') at (3,-2) {6};
\node[vertex] (3'') at (6,-2) {9};
\node[vertex] (3''') at (9,-2) {12};
\node[vertex] (3'''') at (12,-2) {15};
\draw[-latex,thick,dashed] (1) to (1');
\draw[-latex,thick,dashed] (1') to (1'');
\draw[-latex,thick,dashed] (1) to (2');
\draw[-latex,thick,dashed] (1') to (2'');
\draw[-latex,thick] (1'') to (2''');
\draw[-latex,thick,dashed] (1) to  (3');
\draw[-latex,thick,dashed] (1') to  (3'');
\draw[-latex,thick,dashed] (1'') to  (3''');
\draw[-latex,thick] (2) to  (3'');
\draw[-latex,thick,dashed] (2') to  (3''');
\draw[-latex,thick,dashed] (2'') to  (3'''');
\draw[-latex,thick] (3) to  (3');
\draw[-latex,thick,dashed] (3') to  (3'');
\draw[-latex,thick,dashed] (3'') to (3''');
\draw[-latex,thick,dashed] (3''') to (3'''');
\end{tikzpicture}}
     \end{subfigure}
     \hfill
     \begin{subfigure}[b]{0.475\textwidth}
         \centering
         \subcaptionbox{Maximum Flow for Updated Evacuation Plan \label{fig:evacu_update}}{
\begin{tikzpicture}[scale = 0.42]
\tikzset{vertex/.style = {shape=circle,draw,minimum size=2em}}
\node[vertex] (1) at (0,2) {1};
\node[vertex] (1') at (3,2) {4};
\node[vertex] (2) at (0,0) {2};
\node[vertex] (2') at (3,0) {5};
\node[vertex] (2'') at (6,0) {8};
\node[vertex] (3) at (0,-2) {3};
\node[vertex] (3') at (3,-2) {6};
\node[vertex] (3'') at (6,-2) {9};
\node[vertex] (3''') at (9,-2) {12};
\node[vertex] (3'''') at (12,-2) {15};
\draw[-latex,thick,dotted] (1) to (1');
\draw[-latex,thick,dashed] (1) to (2');
\draw[-latex,thick,dotted] (1') to (2'');
\draw[-latex,thick,dashed] (1) to  (3');
\draw[-latex,thick,dotted] (1') to  (3'');
\draw[-latex,thick] (2) to  (3'');
\draw[-latex,thick,dashed] (2') to  (3''');
\draw[-latex,thick,dotted] (2'') to  (3'''');
\draw[-latex,thick] (3) to  (3');
\draw[-latex,thick,dashed] (3') to  (3'');
\draw[-latex,thick,dotted] (3'') to (3''');
\draw[-latex,thick,dotted] (3''') to (3'''');
\end{tikzpicture}}
     \end{subfigure}
        \caption{Maximum flow (dashed and dotted) for an initial (left) and updated (right) evacuation plan. Green arcs have the same flow between solutions while blue arcs have decreased flow.}
        \label{fig:wten_evac_update_ex}
\end{figure}

Efficiency and correctness of these updates transfer quite readily from Theorems \ref{thm:efficiency} and \ref{thm:correctness}. We assume that $f_A$ also represents an upper bound on the cost of computing $f_{ij}(t)$ for an arc $(i,j) \in A$ and $t \in [T]_0$ with respect to the new fire model.

\begin{corollary}\label{cor:efficiency}
     Let $G_W=(N_W,A_W)$ be a WTEN based on a fire set $F(t)$ for $t\in [T]_0$, and let $x_{ij}(t)$ denote arc flows in a maximum flow solution. Further, let $t_{\text{fire}}$ and $t_{\text{reopt}}$ be time instances satisfying $t_{\text{reopt}}\leq t_{\text{fire}} \leq T$, and let $F'(t)$ be a fire set for $t\in t_{\text{fire}},\dots,T$ with $F(t)\subset F'(t)$.

    The (\ref{WTEN_LP}) with fixed flows $x_{ij}(t)$ for $t < t_{\text{reopt}}$ and based on $F'(t)$ can be set up and solved in time $O(T|A|\cdot (T|N| + f_A))$ in the arithmetic model of computation.
\end{corollary}

\begin{proof}
    Construction of the updated network $G'_W=(N'_W,A'_W)$  over the same time horizon differs only in a few points from the original WTEN $G_W=(N_W,A_W)$. By $F(t)\subset F'(t)$, $G'_W$ is a subnetwork of $G_W$, i.e., $N'_W\subset N_W$ and $A'_W\subset A_W$. Its nodes, arcs, and capacities can be efficiently constructed from the original TEN as described in Lemma \ref{lem:setupWTEN}, restricting to nodes and arcs in $G_W$ for time instances $t_{\text{reopt}},\dots, T$.

    The fixation of flows $x_{ij}(t)$ for $t < t_{\text{reopt}}$ can be implemented through specification of a new set of source nodes 
    \begin{equation}
        S'_W=\{j(t') \in N_W: (i(t),j(t'))\in A_W, x_{i(t)j(t')}>0, t'=t +\lambda_{i(t)j(t')}, t < t_{\text{reopt}}\}
    \end{equation}
    with associated supply
    \begin{equation}
     \sum\limits_{t=0}^{t_{\text{reopt}}-1} \sum\limits_{(i(t),j(t'))\in A_W} x_{i(t)j(t')}.     
    \end{equation}
   
    One then replaces the number of time instances $|[T+1]_0|$ with $|[T+1-t_{\text{reopt}}]_0|$, shifting all information for time instances $t \geq t_{\text{reopt}}$ to time instances $t-t_{\text{reopt}}$ instead, and solve the new, smaller (\ref{WTEN_LP}).

    The computational effort for the construction of $S'_W$ is in $O(T|A|)$, as each arc in $A_W$ can only contribute supply to one potential source node. Reducing the time horizon and relabeling information is in $O(T(|A|+|N|)$. The complexity of the integration of the new wildfire information and search for a maximum flow transfer immediately from Theorem \ref{thm:efficiency}. \qed
\end{proof}

The only differences between the constructed subnetwork $G'_W=(N'_W,A'_W)$ of the WTEN $G_W=(N_W,A_W)$ and Definition \ref{def:wten} are in the set of source nodes and the holdover arcs: holdover arcs $(A^H_W)' \subset A'_W$ exist for the original source nodes instead of the new source nodes $S'_W$; and there may be multiple source nodes in $S'_W$ that are copies, for later time instances, of the same source node in $S_W$.  This does not affect the previous arguments in Theorem \ref{thm:correctness}.

\begin{corollary}\label{cor:correctness}
     Let $G_W=(N_W,A_W)$ be a WTEN based on a fire set $F(t)$ for $t\in [T]_0$ satisfying $F(t)\subset F(t+1)$, and let $x_{ij}(t)$ denote arc flows in a maximum flow solution. Further, let $t_{\text{fire}}$ and $t_{\text{reopt}}$ be time instances satisfying $t_{\text{reopt}}\leq t_{\text{fire}} \leq T$, and let $F'(t)$ be a fire set for $t\in t_{\text{fire}},\dots,T$ with $F(t)\subset F'(t)$ and $F'(t)\subset F'(t+1)$.

    The construction of an updated evacuation plan, with fixed flows $x_{ij}(t)$ for $t < t_{\text{reopt}}$ and based on $F'(t)$, displayed in Figure \ref{fig:flow_chart_update}, terminates with a maximum flow, and minimal time horizon over which this flow is possible.
\end{corollary}

\section{Illustrative Applications}\label{sec:test_cases}

As a proof of concept, we compute evacuation plans and updates using our algorithm for three locations: Lyons,  Colorado; Butte County, California; and Alexandroupolis, Greece. The locations for sources and sinks are based on historic fire data or chosen to showcase various capabilities of our algorithm. 

For the specified supply and demand values, we estimate the population of the location and specified nodes. Tolerance values are chosen to test the strength of our algorithm on a finely detailed network, while maintaining efficient overall running times. The time instance of a fire change, $t_\text{fire}$, and the time instance where evacuation crews can implement changes $t_\text{reopt}$, are chosen early on and close to each other, to showcase a quick adjustment in the evacuation plan. For each update, one of the following happens: fires grow more quickly, another fire starts close to the original fire, or another fire starts across town. Note that all of the above assumptions are for user input, so experts can adjust them based on their domain knowledge. We assume our time instances to be minutes, as this allows a fine-grain time-lapse to capture the nature of wildfire spread and the time in which information is relayed among evacuation crews.

\begin{table}[h!]
    \centering
    \begin{tabular}{ |p{4cm}|p{2cm}|p{2cm}|p{2cm}|p{3cm}|}
 \hline
   Area Name& \centering Tolerance & \centering Number of Nodes &  \centering Number of Edges &
\multirow{1.5}{3cm}{\centering Avg. Time to Construct Original Network}
 \\
 \hline
 Lyons, CO & \centering 10 meters  & \centering 107   & \centering 265 & \multicolumn{1}{|c|}{$\sim$ 1 second}\\
 \hline
  Butte County, CA& \centering 10 meters& \centering 12,874 & \centering 31,178 & \multicolumn{1}{|c|}{$\sim$  53 seconds}\\
 \hline
 Butte County, CA& \centering 50 meters& \centering 7,637   & \centering 19,611 & \multicolumn{1}{|c|}{$\sim$ 58 seconds}\\
 \hline
Paradise, CA& \centering 50 meters &\centering 917   & \centering 2,163 & \multicolumn{1}{|c|}{$\sim$ 7 seconds}\\
 \hline
 Alexandroupolis, Greece& \centering 10 meters &\centering 1,060 & \centering  2,749 &\multicolumn{1}{|c|}{$\sim$ 4 seconds} \\
 \hline
\end{tabular}
    \vspace{1mm}
    \caption{Summary of data for the original dynamic networks of each case.}
    \label{tab:orig_net_data}
\end{table}

Table \ref{tab:orig_net_data} provides a summary of the sizes and the time required to construct the original dynamic network for each test case. The computation times are based on an average of twenty runs for each case on a machine with 16 gigabytes of RAM, a 64-bit operating system, an 8 core CPU, and 8.1 gigabytes of GPU memory. The networks contain all necessary node and arc information required for the algorithm to construct the WTEN (with wildfire information to be added later). Travel times for arcs come from the public data and are converted to minutes. Dynamic networks for all cases are constructed using a tolerance of 10 meters unless stated otherwise.

For each case, we place an approximated fire in either a historic fire location or, in the case of Lyons, Colorado, in an area of high fire danger \citep{oh-11}. We begin by constructing an initial evacuation plan based on these fire locations. Then, we create a change (or multiple changes) in the fire at a specified time instance. Finally, we update the initial evacuation plan to account for the change. For case 2 (Butte County), we focus on an evacuation of the town of Paradise, California. An initial evacuation plan is run over the large Butte County network, while the update portion of our approach is run only on a network for the town itself. This lets us discuss scalability and computation times for the two vastly different problem sizes.

At the end of this section, we summarize the computation times of each case for the construction of the initial WTEN, fire model, and updates. In addition, we discuss deliverables that we are able to provide to practitioners.

\subsection{Case 1: Lyons, Colorado}

The first location is Lyons, Colorado, a town located at the base of the foothills of Colorado with few roads in and out of the town. The fire danger in this town is high due to the location \citep{oh-11}. We chose to place the fire near the northeastern part of the town, which is a large area of open land and is under considerable fire watch. The sources and sinks, including their supply/demand values, are based on the population of the town and showcase a difference between the initial and updated evacuation plans.

We give each of the two source nodes a supply of 4,000 people. Lyons has a population of roughly 6,000 people \citep{oh-11}. This overestimation in supply is done to visualize an aversion to fire risk in the evacuation routes chosen.
Everyone is able to evacuate with a total flow value of 8,000 within 24 minutes.

To represent a dangerous real-time change in the fire, we model the fire becoming larger than initially anticipated at time instance $t_\text{fire}=11$. We assume the evacuation crews are able to react to this change at time instance $t_\text{reopt}=4$. That is, they will have 7 minutes during which people can be redirected from initial plan routes before the actual change in fire. The difference between $t_\text{fire}$ and $t_\text{reopt}$ is based on user input and thus, in practice, can be adjusted based on expert knowledge.

Figure \ref{fig:lyons_evac_plans} shows all edges that had flow (green) for the original and updated evacuation plans during time instances $0 \leq t \leq T=24$. The fire set, $F(T)$, (yellow) in each sub-figure is the fire at the final time instance $t=24$. In the updated evacuation plan, the fire has grown much faster as seen by the larger $F(T)$ in Figure \ref{fig:update_evac_lyons}. Therefore, the edge capacities decrease in larger quantities for time instances after $t_\text{fire}=11$. This forces evacuees to use other (previously not optimal) roads that still allow for safe travel to the designated sinks; see Figure \ref{fig:lyons_evac_plan_change}.

\begin{figure}[htb]
     \centering
     \begin{subfigure}[b]{0.475\textwidth}
         \centering
         \subcaptionbox{Original Evacuation Plan\label{fig:orig_evac_lyons}}{
         \includegraphics[scale = 0.35]{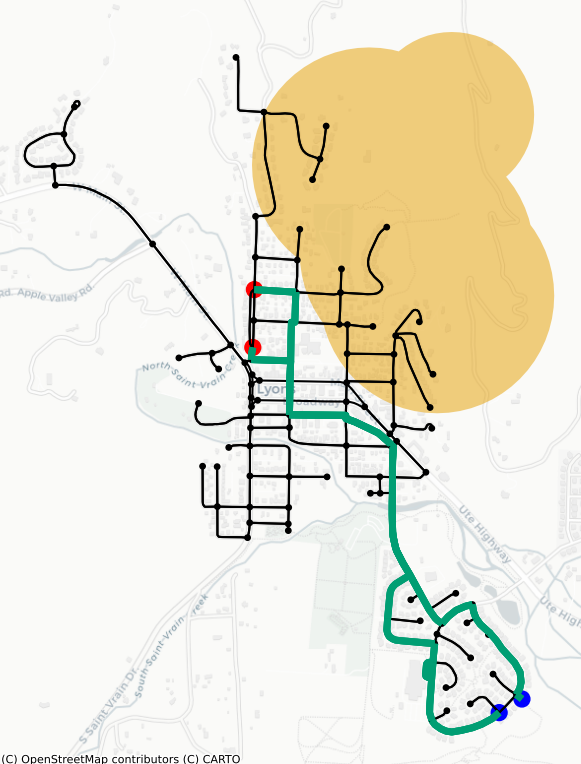}
         }
     \end{subfigure}
     \hfill
     \begin{subfigure}[b]{0.475\textwidth}
         \centering
         \subcaptionbox{Updated Evacuation Plan \label{fig:update_evac_lyons}}{
         \includegraphics[scale = 0.35]{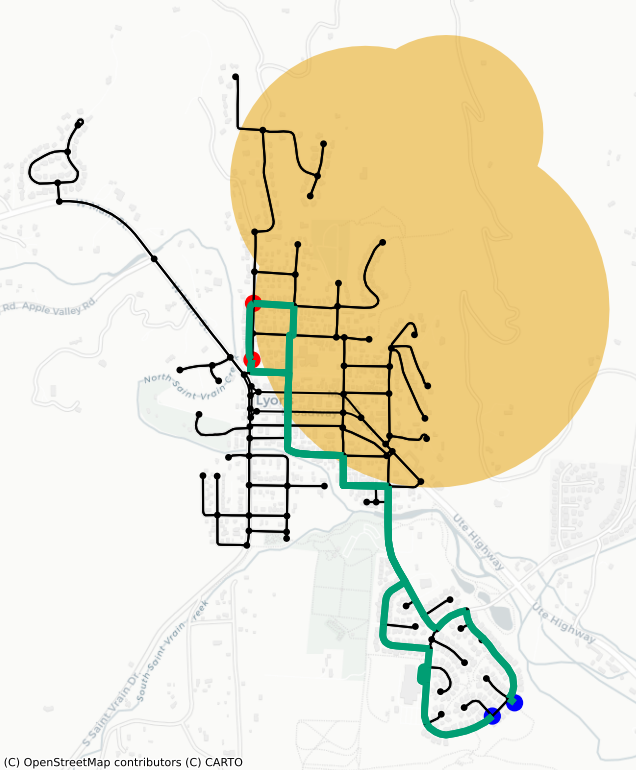}
         }
     \end{subfigure}
        \caption{Initial (left) and updated (right) evacuation plans for Lyons, Colorado, with a final fire set $F(T)$ in yellow. Edges with flow at some point during the evacuation from sources (red) to sinks (blue) are shown in green.}
        \label{fig:lyons_evac_plans}
\end{figure}

 A deviation from the initial evacuation plan, shown in Figure \ref{fig:lyons_evac_plan_change}, starts at time instance $t=9$. It shows edges with flow at some point during $9\leq t \leq 12$ using an alternate route used in the updated plan. This occurs because a road used in the initial plan was engulfed by the fire (starting at time instance $t_\text{fire}=11$); see Figure \ref{fig:update_evac_lyons_zoom}. The change before $t_\text{fire}$ highlights the advantage of quickly relaying information when implementing changes during on-going evacuations.

\begin{figure}[ht!]
     \centering
     \begin{subfigure}[b]{0.475\textwidth}
         \centering
         \subcaptionbox{Original Evacuation  Plan for $9 \leq t \leq 12$\label{fig:orig_evac_lyons_zoom}}{
         \includegraphics[scale = 0.35]{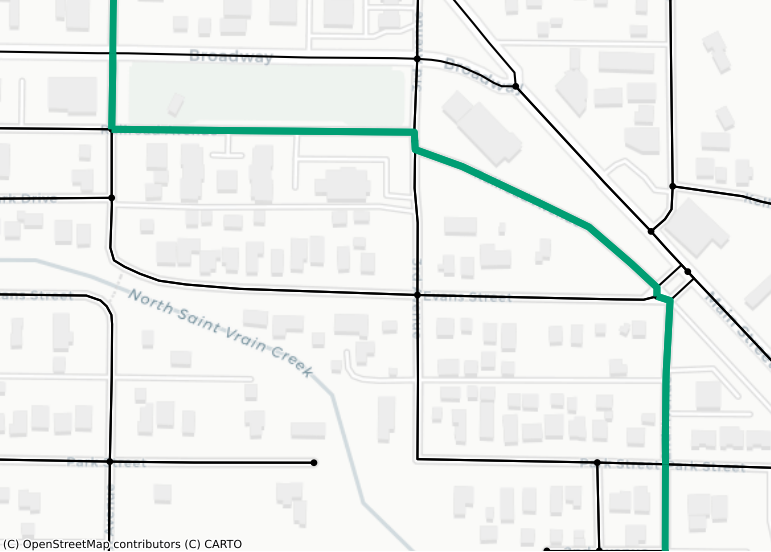}
         }
     \end{subfigure}
     \hfill
     \begin{subfigure}[b]{0.475\textwidth}
         \centering
         \subcaptionbox{Updated Evacuation Plan for $9 \leq t \leq 12$ \label{fig:update_evac_lyons_zoom}}{
         \includegraphics[scale = 0.35]{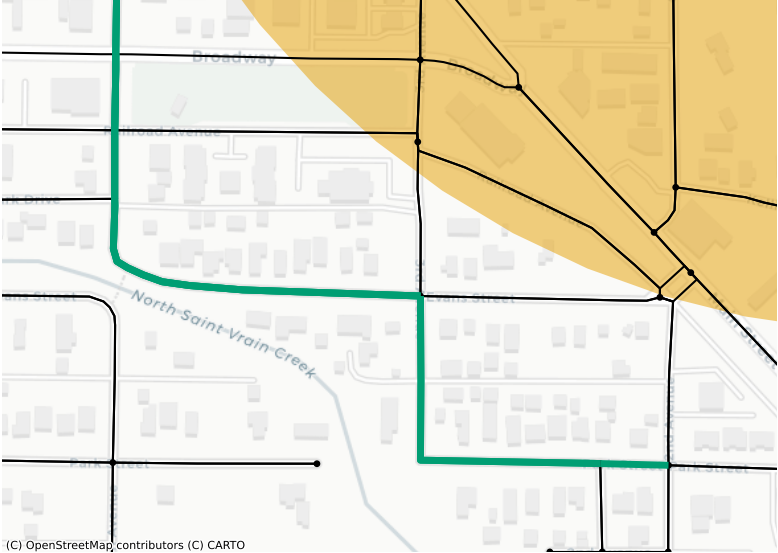}
         }
     \end{subfigure}
        \caption{Initial (right) and updated (left) evacuation plans for Lyons, Colorado with fire set $F(12)$ in yellow. Edges with flow at some point during the given time instances are in green.}
        \label{fig:lyons_evac_plan_change}
\end{figure}

\subsection{Case 2: Butte County, California}

Butte County California is the location of the deadliest wildfire in California's history: the Camp Fire \citep{mlmh-23}. While the Camp Fire mainly affected the residents of Paradise, California, we elect to use the entire road network for Butte County for an initial evacuation plan to provide an example on a large network; see Table \ref{tab:orig_net_data}. The road network used has a coarser tolerance of 50 meters in order to keep scale and running times reasonable, while still maintaining enough detail to accurately capture the evacuation. The fire location is based on the Camp Fire. The sources are nodes in Paradise, California, while the sinks lie outside of the town. The scenario is set up to resemble the situation during the Camp Fire evacuation.

We give one source node a supply of 900 and the other 600. Similarly, the sinks have different demand values of 500 and 1000. The varying supply values represent the density of people at various locations in the town. The varying demand values reflect the limited capacity of the safe locations, such as churches, schools, and stadiums, that people are sent to. 

Due to the large network size for Butte County, we only run an initial evacuation plan using its road network and run an update on the Paradise, California road network. The differences in computational speed can be seen in the tables in Section \ref{sec:runningtime}. Everyone is evacuated within a time horizon of 54 minutes, with a total flow alue of 1,500 over the Butte County network. Figure \ref{fig:orig_butte_evac} shows the evacuation on the entire Butte County road network with a closeup of Paradise.

\begin{figure}[htb]
    \centering
    \begin{tikzpicture}
        \node [draw, red, ultra thick, inner sep=0pt, outer sep=0pt] (image1) at (-3.25,2.5) {\includegraphics[scale = 0.5]{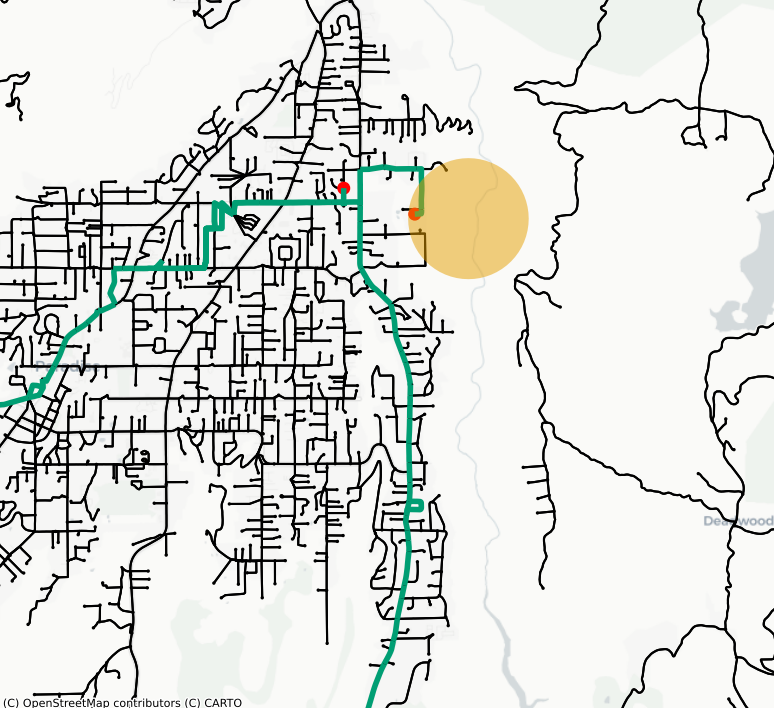}};
        \node[draw, ultra thick, inner sep=0pt, outer sep=0pt] (image2) at (0,0) {\includegraphics[scale = 0.25]{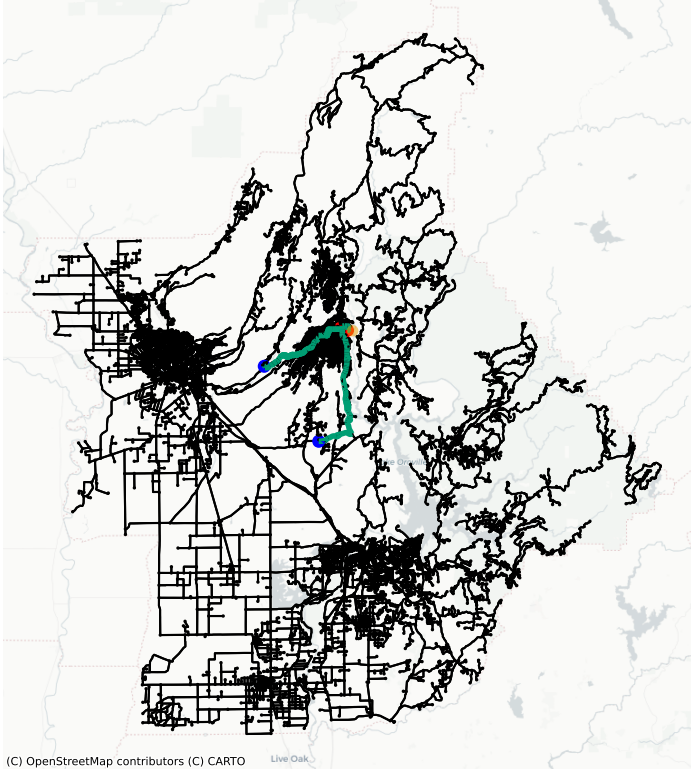}};
        \draw[red, thick] (-0.28,-0.05) rectangle (0.23,0.47);
    \end{tikzpicture}
    \caption{Initial evacuation plan for Paradise, California on Butte County road network with final fire set $F(T)$ in yellow. Edges with flow at some point during the evacuation from sources (red) to sinks (flow) are shown in green.}
    \label{fig:orig_butte_evac}
\end{figure}

We compute an updated evacuation plan restricted to the city of Paradise. Because the labels for nodes and arcs vary between the county and city networks in OpenStreetMaps, we first translate the initial evacuation plan, computed on the county network, to the Paradise road network. To this end, sinks are nodes on the outskirts of the Paradise road network that were visited during the initial evacuation plan on Butte County; sinks (blue) are displayed in Figure \ref{fig:paradise_updated}. 

The fire change takes place at $t_\text{fire} = 10$ with the evacuation crews starting to enact change at $t_\text{reopt} = 5$. While the smaller network size and new sinks can affect the minimum time horizon $T$ necessary, we manually set $T=54$ for consistency between computations.

Figure \ref{fig:paradise_updated} shows edges with flow at some point during the updated evacuation of Paradise, California. The drastic change in the fire at $t_\text{fire}=10$ causes a  previous route (see Figure \ref{fig:orig_butte_evac}) to become unusable, and an alternate path must be found to reach the sink (blue).

\begin{figure}[htb]
    \centering
    \includegraphics[scale = 0.57]{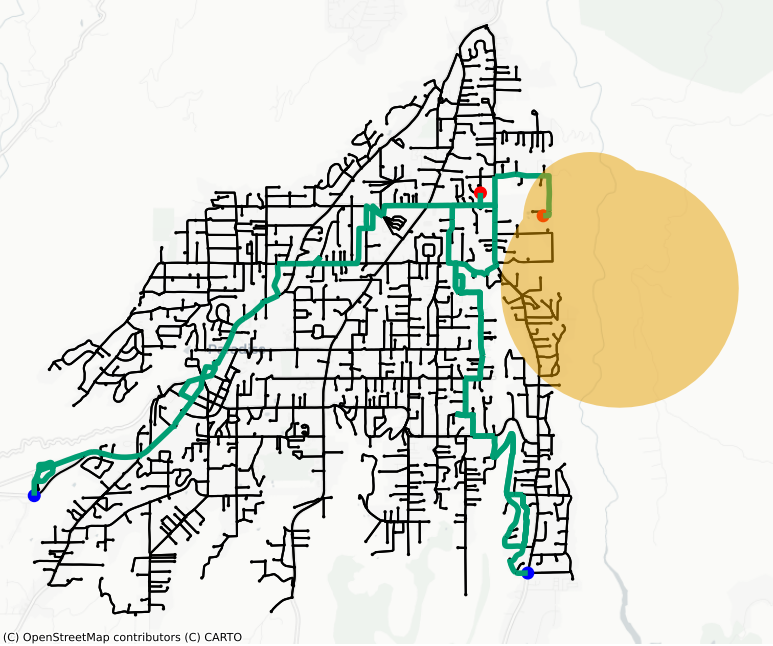}
    \caption{Updated evacuation plan for Paradise, California with final fire set $F(T)$ in yellow. Edges with flow at some point during the evacuation from sources (red) to sinks (blue) are shown in green.}
    \label{fig:paradise_updated}
\end{figure}

\subsection{Case 3: Alexandroupolis, Greece}

In 2023, a series of fires in Greece broke out in a 24-hour time frame, burning a total of roughly 730 square kilometers. It is one of the largest wildfires in European history \citep{eu-24}. Alexandroupolis, Greece was one of the cities affected. Our estimated fire locations are based on this historic fire, with sources selected as nodes close to the fire, in most imminent need of evacuation. A single sink is chosen in the safe south of the city. The supply values are divided among three sources as 1,000, 1,500, and 2,500. A complete evacuation is possible in 48 minutes.

We consider two fire changes which begin at time instances $t_\text{fire} =10$ and $t_\text{fire}=25$, with evacuation crews starting to implement changes at $t_\text{reopt}=5$ and $t_\text{reopt}=20$, respectively. Figure \ref{fig:greece_evac_plans} shows the edges with flow during time instances $0 \leq t \leq T = 48$ for the initial and updated evacuation plans. In the updated plan, we see the removal of several edges, but also the use of some alternative ones. This results in changes to the evacuation plan such as a deviation in the routes used by those from the left-most supply node as well as the use of a single route from the right-most supply node (see Figure \ref{fig:update_evac_greece}). 

\begin{figure}[htb]
     \centering
     \begin{subfigure}[b]{0.475\textwidth}
         \centering
         \subcaptionbox{Original Evacuation Plan\label{fig:orig_evac_greece}}{
         \includegraphics[scale = 0.3]{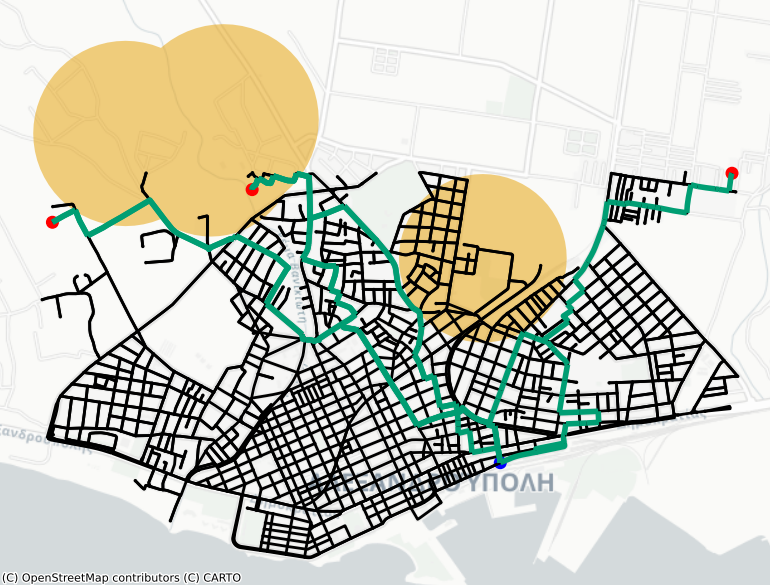}
         }
     \end{subfigure}
     \hfill
     \begin{subfigure}[b]{0.475\textwidth}
         \centering
         \subcaptionbox{Updated Evacuation Plan \label{fig:update_evac_greece}}{
         \includegraphics[scale = 0.325]{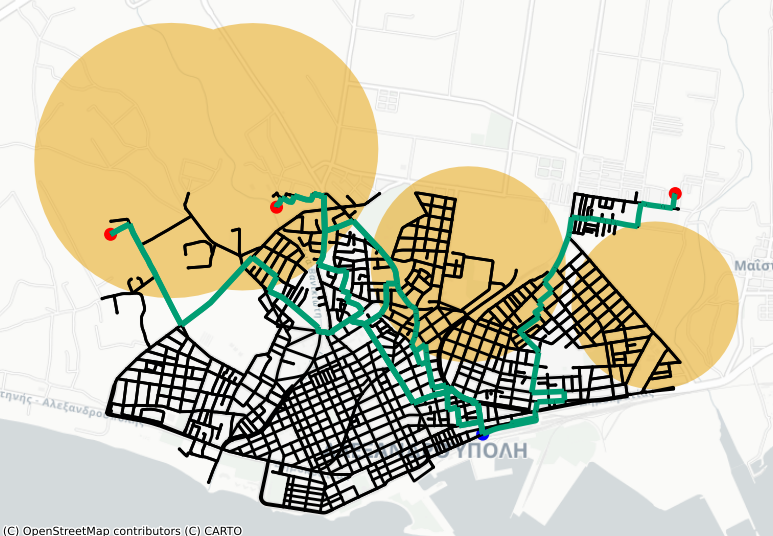}
         }
     \end{subfigure}
        \caption{Initial (right) and updated (left) evacuation plans for Alexandroupolis, Greece with final fire set $F(T)$ in yellow. Edges with flow at some point during the evacuation from sources (red) to sink (blue) are shown in green.}
        \label{fig:greece_evac_plans}
\end{figure}

\subsection{Running Time Summary}\label{sec:runningtime}

We now compare the running times for each case. First, we look at the computation time required to create an initial evacuation plan. Table \ref{tab:sum_init_evac} provides a summary of the running times for each case when constructing the WTEN and modeling the fire for the given time horizons.

The most time-consuming aspect of the algorithm is the wildfire modeling, more precisely, the time it takes to determine if nodes have been overtaken by the fire and how far every edge is from the fire for each time instance. Some computational effort could be saved by looking at previous time instance information to determine which nodes and edges no longer need to be considered for these computations. However, the time saved in practice would be low as most networks are large and the number of nodes and edges overtaken small in comparison to the ones that still need to be checked.

\begin{table}[h!]
    \centering
\begin{tabular}{|p{2.6cm}|p{1.35cm}|p{2.5cm}|p{2.5cm}|p{2.5cm}||p{2.3cm}|}
 \hline
   \raggedright Area Name& \centering Time horizon $T$ &  \centering Fire Model Time & \centering WTEN Construction Time & \centering Maximum Flow Time & \multicolumn{1}{p{2.3cm}|}{\centering Total Wall Time} \\
    \hline

 \raggedright Lyons, CO   & \centering 24   & \centering $\sim$ 6 seconds & \centering $< 1$ second & \centering $< 1$ second &  \multicolumn{1}{|c|}{$\sim$ 7 seconds}\\ 
 \hline
 \raggedright Butte County, CA& \centering 53   & \centering $\sim$ 785 seconds  & \centering  $\sim$ 6 seconds  & \centering $\sim$ 33 seconds &  \multicolumn{1}{|c|}{$\sim$  886 seconds}\\
 \hline
\raggedright Paradise, CA& \centering 53   & \centering $\sim$ 56 seconds & \centering $< 1$ second & \centering $\sim$ 2 seconds & \multicolumn{1}{|c|}{$\sim$  71 seconds}\\
 \hline
\raggedright  Alexandroupolis, Greece& \centering 48 & \centering  $\sim$  99 seconds & \centering $< 1$ second & \centering $\sim$ 3 seconds & \multicolumn{1}{|c|}{$\sim$  108 seconds} \\
 \hline
\end{tabular}
    \caption{Average running times to create initial evacuation plan for each location.}
    \label{tab:sum_init_evac}
\end{table}

Next, we turn to the update portion of the algorithm. For each location we chose an early time instance $t_\text{fire}$ for an update. Once again, the most time-consuming part of the algorithm is building the new fire data. Table \ref{tab:sum_update} summarizes the running times for integrating the new fire data and updating the evacuation plan. We see that the times are similar to the creation of the initial evacuation plan, and that creating a fire set and checking the distance to each arc takes the most time.

\begin{table}[htb]
    \centering
\begin{tabular}{|p{2.6cm}|p{1.35cm}|p{2.5cm}|p{2.5cm}|p{2.5cm}||p{2.3cm}|}
 \hline
   \raggedright Area Name& \centering Time of Fire Change $t_\text{fire}$ &  \centering Updated Fire Model Time & \centering Updated WTEN Construction Time & \centering Updated Maximum Flow Time & \multicolumn{1}{p{2.3cm}|}{\centering Total Update Wall Time} \\
    \hline
 \raggedright Lyons, CO   & \centering 11   & \centering $\sim$  2.6 seconds  & \centering $< 1$ second& \centering $< 1$ second &  \multicolumn{1}{|c|}{$\sim$ 3 seconds }\\ 
 \hline
\raggedright Paradise, CA& \centering  10  & \centering  $\sim$ 48 seconds & \centering $\sim$ 1 second & \centering $\sim$ 2 seconds  & \multicolumn{1}{|c|}{ $\sim$ 54 seconds}\\
 \hline
\raggedright  Alexandroupolis, Greece& \centering 10, 25 & \centering $\sim$ 83 seconds, \quad $\sim$ 62 seconds  & \centering $\sim$ 2 seconds & \centering $\sim$ 3 seconds, \quad  \ $\sim$ 2 seconds  & \multicolumn{1}{|p{2.3cm}|}{\centering $\sim$ 93 seconds, \quad $\sim$ 71 seconds} \\
 \hline
\end{tabular}
    \caption{Average running times to update fire and create updated evacuation plan for each location.}
    \label{tab:sum_update}
\end{table}

\subsection{Output Format}\label{sec:output} 
A key consideration is what information, and in what format, we can provide to practitioners. The direct result of our algorithm are flow values on the WTEN. Recall that the WTEN is an abstract network, formed from public road network information by contracting nodes below a certain tolerance; see Section \ref{sec:data_processing}. The information imported for OpenStreetMaps includes the names of the roads and is retained in the contraction of our network, allowing the transformation of flow values from the arcs in WTEN back into information on the original network. This leads to the ability to provide flow values for the original road network for all time instances: an evacuation plan as a list of roads (road sections) that will be used, and at what times.

\begin{figure}[h!]
     \centering
     \begin{subfigure}[b]{0.3\textwidth}
         \centering
         \subcaptionbox{Updated Evacuation Plan at time instance $t=9$.}{
         \includegraphics[scale = 0.2]{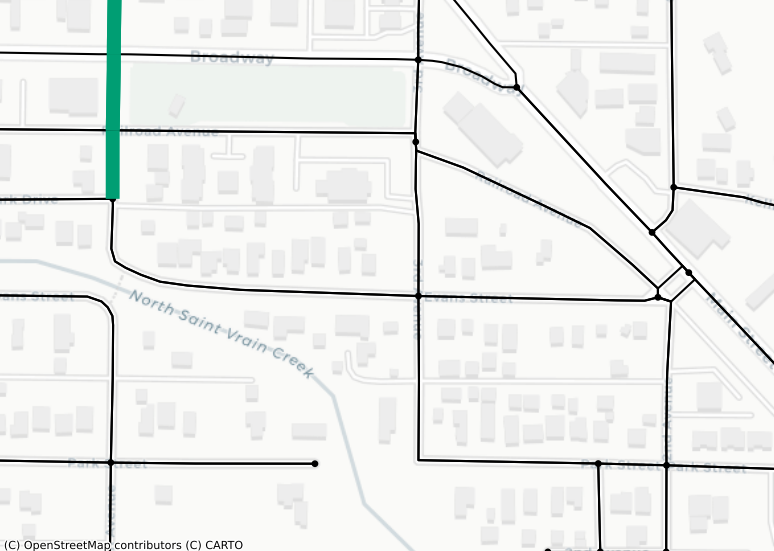}
         }
     \end{subfigure}
     \hfill
      \begin{subfigure}[b]{0.3\textwidth}
         \centering
         \subcaptionbox{Updated Evacuation Plan at time instance $t=10$.}{
         \includegraphics[scale = 0.2]{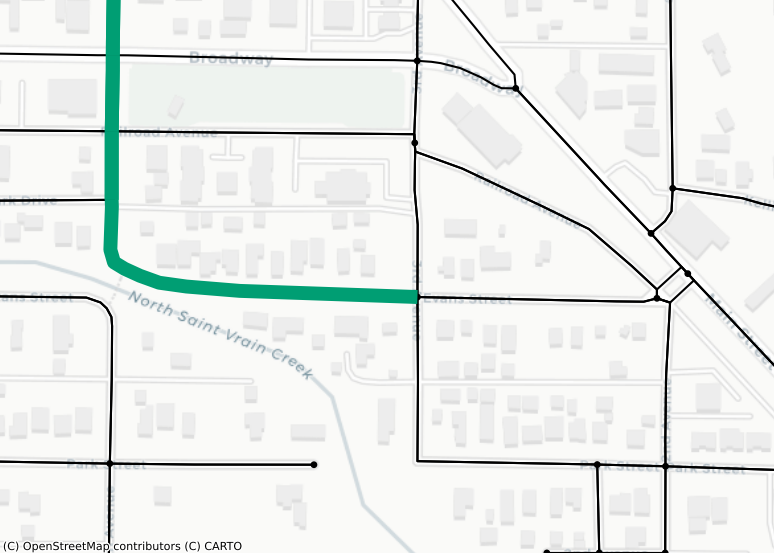}
         }
     \end{subfigure}
     \hfill
     \begin{subfigure}[b]{0.3\textwidth}
         \centering
         \subcaptionbox{Updated Evacuation Plan at time instance $t=11$.}{
         \includegraphics[scale = 0.2]{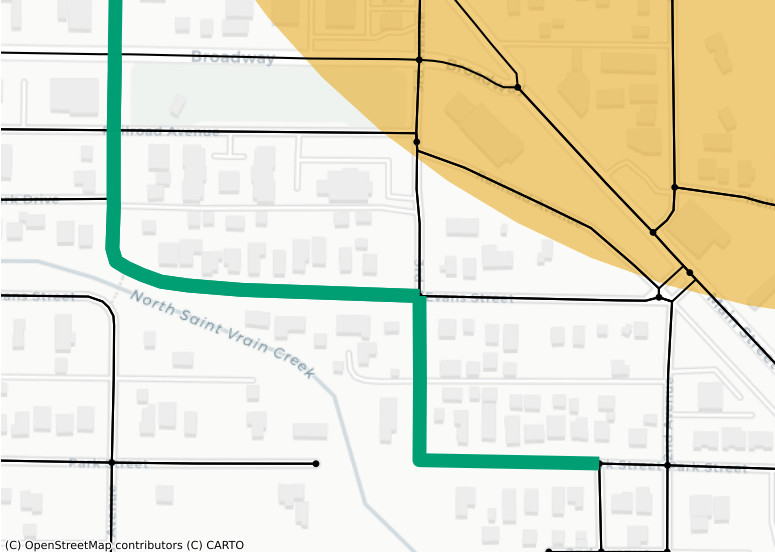}
         }
     \end{subfigure}
        \caption{Edges with flow during the specified time instances are shown in green.}
        \label{fig:lyons_flow_over_time}
\end{figure}
The plots in Section \ref{sec:test_cases} that visualize flow were automatically generated through mapping routines that we developed as part of the software. The WTEN and associated road network is overlayed with a map of the region to visualize the evacuation plan. These plots can also be broken down to show which edges have flow for individual time instances or ranges of time instances. See Figure \ref{fig:lyons_flow_over_time} for an example of flow for time instances $t = 9,10,11$ of the initial evacuation plan for Lyons. We believe that the ability to visualize and map the flow contributes significantly to practical viability.

Additionally, the computed capacities for the WTEN, which are not part of the public data, can also be transferred to the original road network and may be of interest in their own right. Together, this information can serve as a basis for a bottleneck analysis in the road network.

\section{Final Remarks}\label{sec:conclusion}

We present an algorithm for evacuation planning based on a maximum flow formulation on a time-expanded network. Our key contribution is the automated integration of wildfire information into a road network constructed from open-source data from OpenStreetMaps. We developed software for this integration using Python packages OSMNX and Shapely. Through the specification of locations and fire areas, our public code can be used to create evacuation plans based on a predicted fire and then update such a plan as new information on the developing fire becomes available. Our software includes routines to map and visualize these plans.

The computational speed of our algorithm is a notable strength. For small towns, such as Lyons or Paradise, an evacuation plan can be computed in roughly a minute. For a larger city such as Alexandroupolis, this increases only to about two minutes. These running times include pre-processing of the public data and output. The practical efficiency is a key aspect for viability in updating evacuation plans in real-time.

Scalability to very large networks can be seen in our sample run on Butte County, which is more than an order of magnitude larger in size of the network. The distance calculations between the fire and arcs of the network become the computational bottleneck. The other parts of the algorithm, such as construction of the WTEN (after the fire model has been calculated) and the run of a maximum flow algorithm remain highly efficient. Our algorithm remains viable for the computation of an initial evacuation plan even for such large geographic areas. 

There are several natural directions for future work.  The current black box methods used for distance computations between roads and fire sets are a bottleneck in overall algorithm running  time.  It would be beneficial to explore creating in-house methods for these computations, either exact or approximate. One could also test the impact of data extrapolation, i.e., reusing distance information over a certain number of time instances before recomputing.  Other algorithmic improvements to consider include warm starting a combinatorial algorithm for the reoptimization of (\ref{WTEN_LP}) (using new fire data).  This would be based on the initial dual optimal solution and could further decrease the time needed to construct an updated plan. 

Further challenges in evacuating planning include increasing evacuation compliance and ensuring equity \citep{Zhao2021}.  Many measures of equity have been considered within a mathematical optimization model, however, a major hurdle to overcome in all cases is the resulting increase in computational time.  An important question that merits further research is how to best account for equity in evacuation planning without sacrificing model efficiency. A starting point for this could be to incorporate equity measures such as the Kolm Pollak Equally Distributed Equivalent which has been shown to scale well to large problem instances in related settings \citep{horton2024scalable,Kolm1976,MS2020}. One could also consider first introducing equity measures into one key step of the process, for example, the initial assignment portion of the algorithm described in \citep{Lim2012}.

Finally, a computational sensitivity analysis of user inputs would further our understanding of the robustness of the approach.  In particular, it would be interesting to explore the question: ``Do small changes to fire sets, safe locations, or supply and demand, lead to small changes in the flow solution?" Such an analysis could be strengthened using true historic wildfire data, on diverse sets of road networks, and by making use of various methods of fire prediction.
\\\\
\subsection*{Acknowledgments}
This work was supported in part by the Air Force Office of Scientific Research [grants FA9550-21-1-0233 and FA9550-24-1-0240] (Complex Networks); and the National Science Foundation [grant 2006183] (Algorithmic Foundations, Division of Computing and Communication Foundations).

\bibliographystyle{plainnat}

\end{document}